\documentclass[a4paper,11pt]{amsart}

\usepackage{amsmath,amssymb,amsthm,mathtools,hyperref}
\usepackage{color}
\usepackage{empheq}
\usepackage{mathrsfs}
\usepackage{mathtools}

\usepackage{array}
\usepackage{booktabs}
\usepackage{diagbox} 
\usepackage{indentfirst}
\usepackage{cite}   

\makeatletter
\renewcommand{\subsection}{\@startsection{subsection}{2}{\z@}%
	{-3.25ex\@plus -1ex \@minus -.2ex}%
	{1.5ex \@plus .2ex}%
	{\normalsize\centering}}
\makeatother

\makeatletter
\def\@seccntformat#1{\csname the#1\endcsname\quad}
\renewcommand\section{%
	\@startsection{section}{1}{\z@}%
	{-3.5ex \@plus -1ex \@minus -.2ex}%
	{2.3ex \@plus.2ex}%
	{\normalfont\normalsize\bfseries\centering}%
}
\makeatother

\numberwithin{equation}{section} \theoremstyle{plain}
\newtheorem{thm}{Theorem}[section]

\newtheorem{prop}[thm]{Proposition}
\newtheorem{lem}[thm]{Lemma}
\newtheorem{cor}[thm]{Corollary}
\newtheorem{rem}[thm]{Remark}

\makeatletter

\def\({\left(}
\def\){\right)}
\newcommand{\R}{\mathbb{R}}
\newcommand{\CC}{\mathbb{C}}

\DeclareMathOperator{\oi}{\mathbf{i}}

\DeclareMathOperator{\odiag}{\mathrm{diag}}

\makeatother
\title[SDP Feasibility Problems and sos Representation Ranks]{SDP Feasibility Problems and sos Representation Ranks for OT-FKM Type Isoparametric Polynomials}
\author[J. Q. Ge]{Jianquan Ge}
\address{School of Mathematical Sciences, Beijing Normal University, Beijing 100875, P. R. China}
\email{jqge@bnu.edu.cn}
\author[K. Jia]{Kai Jia$^{*}$}
\address{School of Mathematical Sciences, Beijing Normal University, Beijing 100875, P. R. China}
\email{kk@mail.bnu.edu.cn}
\author[Y. Y. Zhao]{Yuyang Zhao}
\address{School of Mathematical Sciences, Beijing Normal University, Beijing 100875, P. R. China}
\email{yyzhao24@mail.bnu.edu.cn}
\thanks {$^{*}$ the corresponding author.}
\date{}
\thanks{J. Q. Ge is partially supported by the NSFC (No. 12571049) and the Fundamental Research Funds for the Central Universities.}
\begin{document}
\maketitle
	
	
\begin{abstract}
	Semidefinite programming (SDP) provides a fundamental framework for studying properties of sum-of-squares (sos) representations of nonnegative polynomials. In this paper we study quartic forms associated with OT-FKM type isoparametric polynomials with g = 4. We characterize the sos property of these forms in terms of the feasibility of an explicit SDP determined by the underlying Clifford system. In the sos cases, we further obtain quantitative rank bounds for sos representations, together with rigidity phenomena when m is at least 3.
\end{abstract}



\section{Introduction}\label{secintro}

A real polynomial $p(x)$ in $n$ variables is called \emph{positive semidefinite}
(\emph{psd} for short) or \emph{nonnegative} if $p(x)\ge 0$ for all $x\in\mathbb{R}^n$;
it is called a \emph{sum of squares} (\emph{sos}) if there exist real polynomials $p_k$
such that $p=\sum_k p_k^2$. Since any \emph{psd} or \emph{sos} polynomial can be made
homogeneous by adding one extra variable (preserving the \emph{psd}/\emph{sos} property),
it is convenient to work with homogeneous polynomials (forms).
For an even degree $d$, we denote by $P_{n,d}$ the cone of \emph{psd} forms of degree $d$
in $n$ variables, and by $\Sigma_{n,d}\subseteq P_{n,d}$ the cone of \emph{sos} forms.
Determining whether a given $p\in P_{n,d}$ belongs to $\Sigma_{n,d}$ is a central topic
in real algebraic geometry.

A central computational tool for \emph{sos} is \emph{semidefinite programming} (SDP).
Parrilo and Lall \cite{Lall03} introduced a powerful framework that converts \emph{sos}
questions into SDPs, and Papachristodoulou \emph{et al.} \cite{Lall09} further developed
algorithmic constructions based on this approach in stability problems for nonlinear systems
with time delays.

A semidefinite program is a convex optimization problem that, in its standard (primal) form,
can be written as
\[
\begin{aligned}
	& \underset{X \in SM(n)}{\emph{minimize}}
	& & \langle C, X \rangle \\
	& \emph{subject\ to}
	& & \langle A_i, X \rangle = b_i,\quad i=1,\ldots,m,\\
	& & & X \succeq 0,
\end{aligned}
\]
where $SM(n)$ denotes the space of real symmetric $n\times n$ matrices,
$\langle C,X\rangle=\mathrm{tr}(C^{T}X)$ is the matrix inner product,
and $A_i,C\in SM(n)$, $b_i\in\mathbb{R}$ are given.

The equalities $\langle A_i,X\rangle=b_i$ define an affine subspace of
$SM(n)$ and are therefore referred to as the \emph{affine constraints}.
An SDP is said to be \emph{feasible} if there exists a matrix
$X\in SM(n)$ satisfying these affine constraints together with the
semidefinite constraint $X\succeq0$; such a matrix $X$ is called a
\emph{feasible solution (or feasible matrix)} of the SDP.
In the present paper we will mainly deal with this feasibility problem.

The \emph{sos} property of a polynomial can be characterized via semidefinite programming.
Indeed, Proposition~\ref{sostoSDP} shows that a form $p(x)$ of degree $2d$ is \emph{sos}
if and only if there exists a symmetric matrix $S\succeq 0$ such that
\[
p(x)=z(x)^T S z(x), \qquad z(x):=\bigl(x^\alpha\bigr)_{|\alpha|\le d}.
\]
Beyond feasibility, this SDP viewpoint also encodes quantitative information on \emph{sos}
representations.  

If $p(x)$ admits an \emph{sos} representation $p(x)=\sum_{k=1}^{N} p_k(x)^2$,
the \emph{rank of the sos representation} is defined by
\[
r:=\dim \operatorname{span}\{p_1,\ldots,p_N\},
\]
that is, the number of linearly independent polynomials among the summands.
Under the SDP characterization in Proposition~\ref{sostoSDP}, such a representation corresponds to
a feasible matrix $S\succeq 0$, and the above rank $r$ coincides with $\mathrm{rank}(S)$.
In particular, the set of all possible ranks of \emph{sos} representations of $p$ can be read off
from the ranks of feasible solutions $S$.  We develop this correspondence systematically in
Subsection~\ref{Ranks via SDP}.

A particularly interesting class of structured \emph{psd} forms arises from
isoparametric geometry in spheres.
A function $f$ on a Riemannian manifold is called \emph{isoparametric} if
$|\nabla f|^2$ and $\Delta f$ are functions of $f$.
These two conditions imply that the regular level sets $M_t:=f^{-1}(t)$ form a family of
parallel hypersurfaces with constant mean curvature (\emph{cf.} \cite{Cecil15,GeQianTangYan25}).
In a unit sphere (more generally, a real space form), this is equivalent to the classical
``constant principal curvatures'' condition;  for background on the classification theory of isoparametric hypersurfaces and its applications,
see \cite{Cecil15,Cecil07,Chi11,Chi13,Chi20,Fa17,Mun,OT75,FKM81,GT23,Miyaoka13,Miyaoka16,
GeQianTangYan25,GT13,QT15,TangYan13,TangXieYan14,Solomon92} and references therein.

A fundamental result of M\"unzner \cite{Mun} asserts that an isoparametric hypersurface
$M\subset \mathbb{S}^{n-1}$ is (an open part of) a regular level set of an isoparametric function
$f=F|_{\mathbb{S}^{n-1}}$, where $F$ is a homogeneous polynomial on $\mathbb{R}^{n}$
satisfying the Cartan--M\"unzner equations
\begin{equation}\label{CMeq}
	\left\{\begin{array}{ll}
		|\nabla F|^2 =g^2|x|^{2g-2}, &\\[2pt]
		\Delta F=\frac{g^2}{2}(m_{-}-m_{+})|x|^{g-2},
	\end{array}
	\right. \quad x\in\mathbb{R}^{n},
\end{equation}
where $g=\deg(F)$ equals the number of distinct principal curvatures, and
$m_{\pm}$ are their multiplicities (with respect to the normal direction $\nabla f/|\nabla f|$).
Moreover, $g\in\{1,2,3,4,6\}$ \cite{Mun}; see also \cite{Fa17} for an independent proof.
The restriction $f=F|_{\mathbb{S}^{n-1}}$ satisfies $|\nabla f|^2=g^2(1-f^2)$ on $\mathbb{S}^{n-1}$,
so $\operatorname{Im}(f)=[-1,1]$.
For $t\in(-1,1)$, the level sets $f^{-1}(t)$ are isoparametric hypersurfaces in $\mathbb{S}^{n-1}$,
and the singular level sets
\[
M_{\pm}:=f^{-1}(\pm1)
\]
are smooth submanifolds of codimension $m_{\pm}+1$, called the \emph{focal submanifolds}.

Starting from an isoparametric polynomial $F$, Ge and Tang \cite{GT23} introduced the following
explicit \emph{psd} forms:
\begin{equation}\label{psdform}
	\left\{\begin{array}{ll}
		G_F^{\pm}(x):=|x|^g\pm F(x)\in P_{n,g}, & g \text{ even},\ g=2,4,6;\\[2pt]
		H_F(x):=|x|^{2g}-F(x)^2\in P_{n,2g},    & g=1,2,3,4,6.
	\end{array}
	\right.
\end{equation}
They completely classified the \emph{sos}/\emph{non-sos} behavior of \eqref{psdform}
for all possible degrees $g$ in accordance with the classification of isoparametric hypersurfaces.
In particular, $H_F$ is always \emph{sos}; this follows from Lagrange's identity,
Euler's formula, and the Cartan--M\"unzner equations \eqref{CMeq}.
For the forms $G_F^{\pm}$, the behavior depends on the degree $g$ and the associated multiplicity pair $(m_+,m_-)$.
In the quartic case $g=4$, the minus form $G_F^{-}=|x|^{4}-F(x)$
admits a direct \emph{sos} representation, whereas the main difficulty lies in the plus form $G_F^{+}$.

In this paper we study the \emph{sos} property and the possible ranks of \emph{sos} representations
 for the plus form
$G_F^{+}$, where $F$ is the quartic
isoparametric polynomial of OT-FKM type determined by a symmetric Clifford system on $\mathbb{R}^{2l}$
(with multiplicity pair $(m_+,m_-)=(m,l-m-1)$). For convenience, throughout the paper we work with
the normalized form
\[
G_F:=\frac{G_F^+}{2}=\frac{|x|^4+F(x)}{2},
\]
which is exactly the \emph{psd} quartic form in \eqref{nonnegativepolyG}.
Writing $n:=2l$, we thus work on $\mathbb{R}^n$ with $n$ even.

The choice of $G_F$ is also motivated by geometry.  Let $f=F|_{\mathbb{S}^{n-1}}$ be the
associated isoparametric function and let $M_\pm=f^{-1}(\pm 1)$ be the focal submanifolds.
Since $G_F=(1+f)/2$ on $\mathbb{S}^{n-1}$, we have
\[
G_F(x)=0\ \text{on}\ \mathbb{S}^{n-1}
\quad\Longleftrightarrow\quad
x\in M_-,
\]
and hence the zero set of $G_F$ in $\mathbb{R}^n$ is exactly the cone over $M_-$.
If $G_F$ is \emph{sos}, say $G_F=\sum_{j} q_j^2$ with quadratic forms $q_j$,
then each $q_j$ vanishes on $M_-$,  forcing the focal cone to be an intersection of finitely many quadrics.  This is closely related to Solomon's study of
quadratic focal varieties and their spectral consequences \cite{Solomon92}, where quadratic
forms vanishing on $M_\pm$ produce explicit Laplace eigenfunctions on the minimal isoparametric
hypersurfaces with eigenvalue $2n$.

Ge and Tang \cite{GT23} completely determined, for all isoparametric polynomials,
whether the associated forms in \eqref{psdform} are \emph{sos} or not.
In particular, for OT-FKM type isoparametric quartics they obtained a definitive qualitative classification of
the \emph{sos}/\emph{non-sos} behavior of the plus form $G_F^{+}$ in terms of the multiplicity
pair $(m_+,m_-)$ and Clifford-algebraic invariants.  However, this qualitative dichotomy does
not address quantitative questions when $G_F$ \emph{is} \emph{sos}, such as the number of
quadratic summands or, more intrinsically, the dimension of the span of these summands.

Our first result gives an explicit SDP characterization for the \emph{sos} property of $G_F$. More precisely, it shows that deciding whether $G_F$ is \emph{sos} can be reduced to the feasibility of a concrete SDP in the matrix variable $B$, whose affine constraints are determined by the Clifford system defining the underlying OT-FKM type isoparametric polynomial.

\begin{thm}\label{thm:SDP-feasible-B}
	Let $G_F$ be the \emph{psd} form in \eqref{nonnegativepolyG} on $\mathbb{R}^n$
	associated with an OT-FKM type isoparametric polynomial $F$.
	Then $G_F$ is \emph{sos} if and only if the following SDP feasibility
	problem admits a solution in the matrix $B$:
	\[
	\begin{cases}
		B\succeq 0,\\[2pt]
		B_{ii}=I_l,\ \ B_{ik}=-B_{ik}^{T},\quad \forall\, 1\le i\neq k\le l,\\[2pt]
		R_iB_{ij}=R_j\quad \forall\, 1\le i,j\le l,
	\end{cases}
	\]
	where $B=(B_{ik})_{i,k=1}^{l}$ is viewed as an $l\times l$ block matrix with blocks
	$B_{ik}\in\mathbb{R}^{l\times l}$, and $R_i$ is defined in \eqref{Define R_q} from the Clifford system.
\end{thm}

The matrix $B$ in Theorem~\ref{thm:SDP-feasible-B} is not merely an auxiliary variable in the SDP characterization. In fact, once a feasible $B$ is obtained, the corresponding \emph{sos} representation of $G_F$ can be written explicitly. More precisely, by Proposition~\ref{R(Q)=R(B-RTR)},
\[
G_F(x)=4\,\widetilde{X}^{T}\bigl(B-R^{T}R\bigr)\widetilde{X},
\]
where $\widetilde{X}$ and $R$ are defined in \eqref{widetilde-X} and \eqref{Define R}, respectively. Therefore any feasible solution $B$ immediately yields an explicit \emph{sos} representation of $G_F$. As a first application of this SDP characterization, we obtain an alternative proof of the complete \emph{sos} classification for $G_F$ associated with OT-FKM type isoparametric polynomials.

\begin{thm}\label{sos thm}
	For all \emph{psd} polynomials $G_F$ in \eqref{nonnegativepolyG} associated with OT-FKM type isoparametric polynomials, the form $G_F$ is \emph{sos} if and only if the multiplicity pair $(m_+, m_-)=(1, k)$, $(2, 2k - 1)$, $(3, 4)$, $(4, 3)^I$ (of indefinite class), $(5, 2)$ or $(6, 1)$ for any $k \in \mathbb{N}^+$.
\end{thm}

More importantly, the SDP viewpoint also allows us to go beyond the mere existence of an \emph{sos} representation and study the possible ranks of such representations. This is essentially different from earlier approaches, which usually prove that $G_F$ is \emph{sos} by constructing one explicit representation (for example, via Lagrange's identity), but do not describe the full range of attainable \emph{sos} representation ranks. By relating \emph{sos} representation ranks to the ranks of feasible SDP matrices through the framework developed in Subsection~\ref{Ranks via SDP}, we obtain the following complete description.

\begin{thm}\label{rank thm}
	Let $G_F$ be the \emph{psd} polynomial of the form \eqref{nonnegativepolyG}
	associated with an OT-FKM type isoparametric polynomial, and assume that $G_F$ is \emph{sos}.
	For any \emph{sos} representation of $G_F$, let $r$ denote its rank
	(i.e., the dimension of the span of the quadratic summands).
	Write the multiplicity pair as $(m_+,m_-)=(m,l-m-1)$.
	\begin{enumerate}
		\item If $(m_+,m_-)=(1,k)$ with $k\in\mathbb{N}^+$, then $l=k+2\geq 3$ and
		\[
		l-1\le r\le \frac{l(l-1)}{2}.
		\]
		\item If $(m_+,m_-)=(2,2k-1)$ with $k\in\mathbb{N}^+$, then $l=2k+2\geq 4$ and
		\[
		l-2\le r\le \frac{l(l-2)}{4}.
		\]
		\item If $(m_+,m_-)=(3,4)$, $(4,3)^I$, $(5,2)$ or $(6,1)$, then the rank is unique and equals
		\[
		r=8-m.
		\]
	\end{enumerate}
	Moreover, in cases \textnormal{(1)} and \textnormal{(2)}, the upper bound can be attained,
	for instance by the explicit \emph{sos} representations obtained from Lagrange's identity, whereas the lower bound can be attained if and only if $l=4$ or $l=8$.
\end{thm}

In Theorem~\ref{rank thm}, the feasible matrices $B$ corresponding to the extremal cases
can be written explicitly once a representative Clifford system is fixed.
For cases \textnormal{(1)} and \textnormal{(2)}, the upper bounds are attained,
for instance, by the matrices $B(1,l)$ and $B(2,l)$ defined in
\eqref{B(1,l)} and \eqref{B(2,l)}, respectively, which arise as feasible
solutions of the associated SDP for the chosen Clifford system.
The lower bounds in these two cases occur when $l=4$ and $l=8$, corresponding
to the matrices $B(2,4)$ and $B^{(6)}$ (defined in \eqref{Define B^6}),
respectively.
In case \textnormal{(3)}, the feasible matrix is always $B^{(6)}$; in fact,
for each of the four multiplicity pairs, it is the unique feasible solution
of the SDP.

An \emph{sos} representation $G_F=\sum_{j=1}^r q_j^2$ produces $r$ linearly independent quadratic
forms vanishing on $M_-$.  By Solomon's result \cite{Solomon92}, such quadratic forms give rise
to Laplace eigenfunctions with eigenvalue $2n$ on the minimal isoparametric hypersurfaces, and hence
Theorem~\ref{rank thm} provides explicit lower bounds for the dimension of the corresponding
eigenspace.  

On the other hand, there is a close connection between \emph{sos} representations of $G_F$
and orthogonal multiplications.  In the OT-FKM type case with $g=4$, the existence of an \emph{sos}
representation of $G_F$ implies the existence of an orthogonal multiplication
\[
T:\mathbb{R}^l\times\mathbb{R}^l\longrightarrow \mathbb{R}^{m+r}
\quad\text{of type }[l,l,m+r],
\]
naturally associated with the underlying Clifford system (see \cite{GT23}).
The existing results provide such a multiplication for some $r$ (equivalently, for some target
dimension $m+r$), but do not determine the possible values of $r$.
Our rank theorem fills this gap: it determines the admissible ranks $r$ of
\emph{sos} representations of $G_F$, and therefore yields corresponding quantitative constraints on the
target dimension $m+r$ of the associated orthogonal multiplications.  In particular, for $m\ge 3$
the rank is uniquely determined and satisfies $m+r=8$, which pins down the target dimension.

\section{Preliminaries}\label{secpre}
All discussions on the OT-FKM type isoparametric polynomial in this paper are based on the following proposition, which transforms the \emph{sos} problem into the feasibility of an SDP problem.
\begin{prop}\label{sostoSDP}
	Let $p(x)$ be a nonnegative polynomial of degree $2d$ in $n$ variables. Then $p(x)$ is \emph{sos} if and only if the following SDP is feasible, i.e., there exists a positive semidefinite matrix $S$ satisfying
	\begin{equation*}
		\begin{cases}
			&S\succeq 0, \\
			&p(x) = z(x)^T S z(x), 
		\end{cases}
	\end{equation*}
	where $z(x)=\bigl(x^\alpha\bigr)_{|\alpha|\le d}$ is the vector of all monomials in $x_1, \ldots, x_n$ of degree at most $d$.
	
\end{prop}

\begin{proof}
	(Necessity): Assume that $p(x)=\sum_{k=1}^{N}p_k(x)^2$ is \emph{sos}.
	Since $\deg p=2d$, each $p_k$ has degree at most $d$.
	Thus, we can let $V_k$ be the vector such that $V_k^T z(x) = p_k(x)$ for $1 \leq k \leq N$, and define $V := (V_1, \cdots, V_N)^T$. It is obvious that the matrix $V^T V$ is positive semidefinite and satisfies $p(x) = z(x)^T (V^T V) z(x)$. Thus, we can take $S = V^T V$.
	
	(Sufficiency): Assume that $p(x) = z(x)^T S z(x)$ with $S \succeq 0$. Then there exists a matrix $V$ such that $S = V^T V$. Hence $p(x) = z(x)^T V^T V z(x) = \| V z(x) \|^2$ is \emph{sos}, where $\|\cdot\|$ denotes the Euclidean norm.
\end{proof}

\begin{rem}\label{rem of sostoSDP}
	For certain special polynomials $ p(x) $, the number of monomials in $ z(x) $ can be reduced to simplify the problem. For instance, if $ p(x) $ is a homogeneous polynomial, taking $ z(x) $ to be all monomials of degree exactly $ d $ is sufficient to obtain the conclusion of Proposition \ref{sostoSDP}.
\end{rem}

Recall that an OT-FKM type isoparametric polynomial is defined as (\textit{cf.} \cite{OT75, FKM81})
\begin{equation}\label{FKM isop. poly.}
	F(x) = |x|^4 - 2\displaystyle\sum_{\alpha = 0}^{m}{\langle
	P_{\alpha}x,x\rangle^2},  \quad x\in\mathbb{R}^{2l},
\end{equation}
where $\{P_0,\cdots,P_m\}$ is  a symmetric Clifford system on $\mathbb{R}^{2l}$, i.e., $P_{\alpha}$'s are symmetric matrices satisfying $P_{\alpha}P_{\beta}+P_{\beta}P_{\alpha}=2\delta_{\alpha\beta}I_{2l}$.
Then the multiplicity pair is $(m_+, m_-)=(m, l-m-1)$. Two Clifford systems $\{P_0,\cdots,P_m\}$ and $\{Q_0,\cdots,Q_m\}$ on $\mathbb{R}^{2l}$ are called \emph{algebraically equivalent} if there exists $A\in O(\mathbb{R}^{2l})$ such that $Q_\alpha = AP_\alpha A^{T}$ for all $\alpha \in\{0, \cdots, m\}$. They are called \emph{geometrically equivalent} when there exists $B\in O(\mathrm{Span}\{P_0, \cdots,P_m\})$ such that $\{Q_0,\cdots,Q_m\}$ and $\{B(P_0),\cdots,B(P_m)\}$ are algebraically equivalent, which give two isoparametric polynomials that are congruent under an orthogonal transformation of $\mathbb{R}^{2l}$.
	
From now on, we write $G_F=G_F^+/2$ for simplicity. Then
\begin{equation}\label{nonnegativepolyG}
	G_F(x)=(F(x)+|x|^4)/2=|x|^4-\sum_{\alpha=0}^m\langle P_\alpha x,x\rangle^2.
\end{equation}

Let $n:=2l$. In order to transcribe $n$-variable  \emph{psd} polynomial $G_F$ into quadratic forms, we define $X$ and $\widetilde{P}_{\alpha}$ as $\bar{n}:=\frac{n(n+1)}{2}$ dimensional column vectors satisfying
\begin{align*}
	&X :=(x_1^2, x_2^2, \cdots, x_n^2, x_1x_2, \cdots, x_ix_j, \cdots, x_{n-1}x_n)^T, & \quad 1 \leq i<j \leq n, \\
	&\widetilde{P}_{\alpha}^{T}X  :=\langle P_\alpha x,x\rangle= \sum_{i=1}^n P^{\alpha}_{ii}x_i^2+2\sum_{1\leq i<j \leq n} P^{\alpha}_{ij}x_ix_j, & \quad 0 \leq \alpha \leq m,
\end{align*}
where $P^{\alpha}_{ij}$ is the $(i,j)$-entry of $P_{\alpha}$.

Let $D$ be a $\bar{n} \times \bar{n}$ matrix that has the $n \times n$ all-ones matrix in its upper-left block and zeros everywhere else, and $\widetilde{P}:=\sum_{\alpha=0}^m \widetilde{P}_{\alpha}\widetilde{P}_{\alpha}^{T}=(\widetilde{P}_{ij,kh})_{\bar{n} \times \bar{n}}$ $(1 \leq i\leq j \leq n, 1 \leq k\leq h \leq n)$ that is a symmetric matrix with 
\begin{equation}\label{barP}
	\widetilde{P}_{ii,kk}=\sum_{\alpha=0}^m P^{\alpha}_{ii}P^{\alpha}_{kk}, \quad \widetilde{P}_{ii,kh}=2\sum_{\alpha=0}^m P^{\alpha}_{ii}P^{\alpha}_{kh}, \quad \widetilde{P}_{ij,kh}=4\sum_{\alpha=0}^m P^{\alpha}_{ij}P^{\alpha}_{kh},
\end{equation}
for $i\neq j, k\neq h$.
Note that the indices $ij$ and $kh$ are ordered as follows: first $\{ii\}_{i=1}^n$, then $\{ij\}_{1 \leq i < j \leq n}$ in lexicographic order. This order, which matches the sequence of $X$, is used for the rows and columns of all $\bar{n} \times \bar{n}$ matrices herein.
Then
\begin{equation*}
	|x|^4  =|x|^2 \cdot |x|^2=X^{T}DX, 
\end{equation*}
\begin{equation*}
	\sum_{\alpha=0}^m\langle P_\alpha x,x\rangle^2  =\sum_{\alpha=0}^m X^{T}\widetilde{P}_{\alpha}\widetilde{P}_{\alpha}^{T}X = X^{T}\widetilde{P}X,
\end{equation*}
\begin{equation}\label{D-P}
	G_F(x) = X^{T}(D - \widetilde{P})X.
\end{equation}

Without loss of generality, we can write the Clifford system $\{P_0,\cdots,P_m\}$ in matrix form under the decomposition $\mathbb{R}^{2l}=E_+(P_0)\oplus E_-(P_0)\cong\mathbb{R}^l\oplus \mathbb{R}^l$, where $E_{\pm}(P_0)$ are the eigenspaces of the eigenvalues $\pm1$ of $P_0$, by
	\begin{equation}\label{Cliffordsys-alg}
		P_0=\begin{pmatrix}
			I_l &  0 \\
			0 & -I_l
		\end{pmatrix}, \quad P_1=\begin{pmatrix}
			0 & I_l    \\
			I_l & 0
		\end{pmatrix}, \quad P_{\alpha+1}=\begin{pmatrix}
			0 & E_\alpha    \\
			-E_\alpha & 0
		\end{pmatrix}, \quad 1\leq\alpha\leq m-1,
	\end{equation}
	where $\{E_1,\cdots,E_{m-1}\}$ generates a Clifford algebra $C_{m-1}$ on $\mathbb{R}^l$, i.e., $E_\alpha$'s are skew-symmetric matrices satisfying $E_\alpha E_\beta+E_\beta E_\alpha=-2\delta_{\alpha\beta}I_l$.
	
	Thus, the entries of matrix $\widetilde{P}=(\widetilde{P}_{ij,kh})_{\bar{n} \times \bar{n}}$ in \eqref{barP} are given by
	\begin{align}
		\widetilde{P}_{ii,kk}&= P^{0}_{ii}P^{0}_{kk}= \begin{cases}
		1, &i,k \leq l ~or ~i,k > l, \\
		-1, &otherwise,
        \end{cases} \label{barP1}\\
		 \widetilde{P}_{ii,kh}&=2P^{0}_{ii}P^{0}_{kh}=0, \label{barP2}\\ \widetilde{P}_{ij,kh}&=4\sum_{\alpha=1}^m P^{\alpha}_{ij}P^{\alpha}_{kh}= \begin{cases}
		 4\sum_{\alpha=1}^m P^{\alpha}_{ij}P^{\alpha}_{kh}, &i,k \leq l ~and ~j,h > l, \\
		 0, &otherwise,\label{barP3}
		\end{cases}
	\end{align}
for $i\neq j, k\neq h$.

Since $G_F$ is a quartic homogeneous form and $X$ consists of all quadratic monomials, the following lemma follows immediately from Proposition \ref{sostoSDP} and Remark \ref{rem of sostoSDP}.

\begin{lem}\label{SDP}
	The \emph{psd} form $G_F$ in \eqref{nonnegativepolyG} on $\mathbb{R}^{n}$ is \emph{sos} if and only if the following SDP is feasible, i.e., there exists a positive semidefinite matrix $Q$ satisfying
	\begin{equation*}
		\begin{cases}
			&Q\succeq 0, \\
			&G_F(x)=X^{T}QX. 
		\end{cases}
	\end{equation*}
\end{lem}

 Let $\mathscr{A}:=\{A \in \mathbb{R}^{\bar{n}\times \bar{n}}: A^{T}=A,\ X^{T}AX=0\}$. Since $G_F(x) = X^{T}(D - \widetilde{P})X$ (see \eqref{D-P}), the lemma states that 
 \[
 G_F\text{ is \emph{sos}} \iff \exists A\in \mathscr{A} \text{ such that } Q = A + D - \widetilde{P} \succeq 0.
 \]

\section{SDP Characterization for the \emph{sos} Property of $G_F$}\label{secSDP}

In this section, we establish the SDP characterization for the \emph{sos} property of $G_F$, and in particular prove Theorem~\ref{thm:SDP-feasible-B}. 
Our main goal is to show that the question whether $G_F$ is \emph{sos} is equivalent to the feasibility problem of an explicit semidefinite program in the matrix variable $B$. 

To achieve this, we introduce several auxiliary matrices and derive a number of structural identities and lemmas. 
Although these preliminary results are obtained here in the course of proving Theorem~\ref{thm:SDP-feasible-B}, they will also play an essential role in the later sections, both in the \emph{sos} classification and in the study of the possible ranks of \emph{sos} representations.

For the remainder of this paper, assume 
\begin{equation}\label{Q=A+D-P}
	Q:=A +D - \widetilde{P}.
\end{equation} 
We establish some relations between the matrices $A=(A_{ij,kh})_{\bar{n} \times \bar{n}}$ and $Q=(Q_{ij,kh})_{\bar{n} \times \bar{n}}$ in Lemma \ref{AP-psd}.
\begin{lem}\label{AP-psd}
$A\in \mathscr{A}$ and $Q$ is positive semidefinite if and only if the following conditions hold:
\begin{enumerate}
    \item \label{condition of A} for indices satisfying $1\leq i,k \leq l$ and $l<j,h \leq n\ (=2l)$, 
    \begin{equation}\label{l2-dimA}
		\left\{\begin{array}{lll}
			A_{ij,ij}=4, \quad A_{ij,ih}=0 ~(h \neq j),\\
			A_{i(i+l),k(k+l)}=4, \quad  A_{i(i+l),kh}=0 ~(h \neq k+l),\\
			A_{ij,kh}=A_{kh,ij}, \\
			A_{ij,kh}=-A_{ih,kj} ~(k \neq i), \\
			A_{ij,kh}=-A_{(j-l)(i+l),kh} ~(j \neq i+l),
		\end{array}
		\right.
	\end{equation}
    \begin{equation}\label{A-P}
    Q_c:=\(Q_{ij,kh}\)_{l^2 \times l^2} =\(A_{ij,kh}-4 \sum _{\alpha=1}^{m} P^{\alpha}_{ij}P^{\alpha}_{kh}\)_{l^2 \times l^2}\succeq 0; 
\end{equation}
    \item for indices satisfying $1\leq i\leq j\leq n$ and $1\leq k\leq h\leq n$ but not satisfying the cases of \emph{(1)},
    \[
    Q_{ij,kh}=0.
    \]
    \label{CONSTofQ}
\end{enumerate}

\end{lem}
\begin{proof}
	For simplicity, we impose the following symmetry conditions on the matrix $A$ for all $i \geq j$ and $k \geq h$:
	$$
	A_{ij,kh} = A_{ij,hk} = A_{ji,kh} := A_{ji,hk}.
	$$
	The same conditions also apply to the matrices $D$, $\widetilde{P}$ and $Q$.
	 Since the monomials \{$x_i^4,\ x_i^3x_j,\ x_i^2x_j^2,\ x_i^2x_jx_k,\ x_ix_jx_kx_h\}_{i, j, k, h \text{ distinct}}$ form a basis for real quartic homogeneous polynomials, $A\in \mathscr{A}$ if and only if, for any $1\leq i,j,k,h\leq n$,
		\begin{equation}\label{AinA}
		\left\{\begin{array}{lll}
			A_{ij,kh}=A_{kh,ij} , \\
			A_{ii,ii}=0=A_{ii,ij}, \\
			A_{ij,ij}+2A_{ii,jj}=0, \\
			A_{ij,ik}+A_{ii,jk}=0 ~ (i,j,k \ \text{distinct}), \\
			A_{ij,kh}+A_{ik,jh}+A_{ih,jk}=0 ~(i,j,k,h \ \text{distinct}).
		\end{array}
		\right.
	\end{equation}
	
	(Necessity): Let the matrix $Q$ be positive semidefinite. This implies that all second-order principal minors of $Q$ are nonnegative. Denote the second-order principal minor of $Q$ formed by rows and columns indexed $ij$ and $kh$ as $$Q\binom{kh}{ij}:=Q_{ij,ij}Q_{kh,kh}-Q_{ij,kh}Q_{kh,ij}.$$ Next, we compute the second-order principal minors $ Q\binom{jj}{ii} $, $ Q\binom{jh}{ii} $, and $ Q\binom{kh}{ij} $ of matrix $ Q $ to determine specific properties of the entries in matrices $ Q $ and $ A $. 
	
	First we have $Q_{ii,ii}=A_{ii,ii}+D_{ii,ii}-\widetilde{P}_{ii,ii}=0$ by \eqref{barP1}, \eqref{Q=A+D-P} and \eqref{AinA}. (The following derivations will repeatedly use \eqref{Q=A+D-P} and \eqref{AinA} without further mention.)
	
	\emph{Case 1:} For any $1\leq i\neq j\leq n$, $Q\binom{jj}{ii} = -Q_{ii,jj}^2 \geq 0$ yields 
	\begin{equation}\label{Qiijj}
	      Q_{ii,jj} = 0.
	\end{equation}

	By \eqref{barP1}, we have
	\begin{align}
		A_{ii,jj}&=Q_{ii,jj}-1+\widetilde{P}_{ii,jj}= \begin{cases}
			-2, &i \leq l ~and ~j>l, \\
			0, &otherwise, 
		\end{cases} 
	\notag\\
		A_{ij,ij}&=-2A_{ii,jj}= \begin{cases}
			4, &i \leq l ~and ~j>l, \\
			0, &otherwise. \label{A_ijij}
		\end{cases}
	\end{align}
	 
	 \emph{Case 2:} For any $1\leq i,j,h\leq n$ with $j\neq h$, $Q\binom{jh}{ii}=-Q_{ii,jh}^2 \geq 0 $ yields 
	 \begin{equation}\label{Qiijh}
	 	Q_{ii,jh}=0.
	\end{equation}	 

	 By \eqref{barP2}, we have
	 \begin{align}
	 	A_{ii,jh}&=Q_{ii,jh}+\widetilde{P}_{ii,jh}= 0, 
	 	\notag\\
	 	A_{ij,ih}&=-A_{ii,jh}=0, \label{A_ab}
	 \end{align}
	where $i,j,h$ are distinct.
	
	\emph{Case 3:} For any $1\leq i\neq j\leq n$ and $1\leq k\neq h\leq n$, $$Q\binom{kh}{ij}=Q_{ij,ij}Q_{kh,kh}-Q_{ij,kh}^2 \geq 0.$$ 

	By \eqref{barP3} and \eqref{A_ijij}, we have 
	\begin{equation*}\label{Q_ijij}
		Q_{ij,ij}=A_{ij,ij}-\widetilde{P}_{ij,ij}= \begin{cases}
			4-4\sum_{\alpha=1}^m (P^{\alpha}_{ij})^2, &i \leq l ~and ~j>l, \\
			0, &otherwise.
		\end{cases}
	\end{equation*}
	\begin{itemize}
		\item If $i<j \leq l$ or $l < i<j$, then $Q_{ij,ij}=0=\widetilde{P}_{ij,kh}$. 
		Hence, $$Q\binom{kh}{ij} = -Q_{ij,kh}^2 = -A_{ij,kh}^2 \geq 0,$$ which implies
	\begin{equation} \label{Q_ijkh for ij}
	    Q_{ij,kh} = A_{ij,kh} = 0\ \text{for }i<j \leq l \text{ 	or } l < i<j.
	\end{equation}
		\item If $k<h \leq l$ or $l < k<h$, by \eqref{Q_ijkh for ij}, then
		\begin{equation}\label{Q_ijkh for kh}
			Q_{ij,kh}=Q_{kh,ij}=0\ \text{for }k<h \leq l \text{ 	or } l < k<h.
		\end{equation}
		\item If $i,k \leq l ~and ~j,h>l$, by \eqref{barP3}, then
		\begin{equation}\label{Q=A-4P}
			Q_{ij,kh}=A_{ij,kh}+D_{ij,kh}-\widetilde{P}_{ij,kh}=A_{ij,kh}-4 \sum _{\alpha=1}^{m} P^{\alpha}_{ij}P^{\alpha}_{kh}.
		\end{equation}
		By \eqref{AinA} and \eqref{Q_ijkh for ij}, we have 
        \begin{equation}\label{jhAntisym}
        A_{ij,kh}=-A_{ik,jh}-A_{ih,kj}=-A_{ih,kj}\ \text{for }i\neq k \text{ and }  j\neq h.
        \end{equation}

	Now, we consider a special case with respect to index $j=i+l$. By the Clifford algebra representation \eqref{Cliffordsys-alg}, we have $P_{i(i+l)}^{1}=1$ and $P_{i(i+l)}^{\alpha +1}=E_{ii}^{\alpha}=0$ for $1\leq\alpha\leq m-1$. Hence, 
	\begin{align*}
		&Q_{i(i+l),i(i+l)}=4-4(P_{i(i+l)}^{1})^2=0,\\
		&Q_{i(i+l),kh}=A_{i(i+l),kh}-4P_{kh}^{1}=\begin{cases}
			A_{i(i+l),k(k+l)}-4, &h=k+l, \\
			A_{i(i+l),kh}, &h \neq k+l.
		\end{cases} 
	\end{align*}
    Thus, 
\begin{equation}\label{Q_ihatikh}
    Q\binom{kh}{i(i+l)} = -Q_{i(i+l),kh}^2 \geq 0 \quad \text{implies} \quad Q_{i(i+l),kh} = 0.
\end{equation}
Consequently,
\begin{equation} \label{A_ihatikh}
    A_{i(i+l),k(k+l)} = 4 \quad \text{and} \quad A_{i(i+l),kh} = 0 \quad \text{for } h \neq k+l.
\end{equation}

\end{itemize}

By \eqref{Q=A-4P}, $Q_c=\(A_{ij,kh}-4 \sum _{\alpha=1}^{m} P^{\alpha}_{ij}P^{\alpha}_{kh}\)_{l^2 \times l^2}$. Since $Q$ is positive semidefinite and $Q_c$ is its principal submatrix, it follows that $Q_c \succeq 0$. Given that
$$\(\sum _{\alpha=1}^{m} P^{\alpha}_{ij}P^{\alpha}_{kh}\)_{l^2 \times l^2} =
\sum _{\alpha=1}^{m}\( P^{\alpha}_{ij}\)_{l^2 \times 1} \(P^{\alpha}_{kh}\)^T_{1 \times l^2}\succeq 0,$$ 
the positive semidefiniteness of $Q_c$  implies $(A_{ij,kh})_{l^2 \times l^2} \succeq 0$ for $i,k \leq l$ and $j,h>l$.  By \eqref{jhAntisym} and \eqref{A_ihatikh}, 
$$A_{ij,(j-l)(i+l)}=-4 \ \text{ for } 1\leq i\neq (j-l)\leq l.$$
Further, by direct calculation, the third-order principal minor of $(A_{ij,kh})_{l^2 \times l^2}$ formed by rows and columns indexed $ij$, $kh$, and $(j-l)(i+l)$ equals $-4(A_{ij,kh}+A_{(j-l)(i+l),kh})^2$. Hence,
\begin{equation} \label{A_hatji}
	A_{ij,kh}=-A_{(j-l)(i+l),kh} ~(j \neq i+l).
\end{equation}
By \eqref{A_ab}, \eqref{jhAntisym}, \eqref{A_ihatikh} and \eqref{A_hatji}, we have
\begin{equation}\label{A_ijkh=-A_ihkj}
	A_{ij,kh}=-A_{ih,kj}\ \text{for }1\leq i\neq k\leq l \text{ and }  l< j, h\leq n.
\end{equation}

In summary, equations \eqref{AinA}, \eqref{A_ijij}, \eqref{A_ab}, \eqref{A_ihatikh}, \eqref{A_hatji} and \eqref{A_ijkh=-A_ihkj} collectively yield condition \eqref{condition of A}, while equations \eqref{Qiijj}, \eqref{Qiijh}, \eqref{Q_ijkh for ij} and \eqref{Q_ijkh for kh} establish condition \eqref{CONSTofQ}. 
Thus we complete the proof of necessity.
	
	(Sufficiency): Assume $A$ and $Q$ satisfy \eqref{condition of A} and \eqref{CONSTofQ}. It follows that $Q$ is supported on the principal submatrix corresponding to $Q_c$, with all other entries being zero. Hence, the positive semidefiniteness of $Q_c$ guarantees the positive semidefiniteness of $Q$.

	On the other hand, by \eqref{barP1}--\eqref{Q=A+D-P}, $Q$ satisfies \eqref{CONSTofQ} if and only if 
	\begin{align*}
		A_{ii,kk}&=  \begin{cases}
			0, &i,k \leq l ~or ~i,k > l, \\
			-2, &otherwise,
		\end{cases} \\
		A_{ii,kh}&=0, \\
		A_{ij,kh}&= 
		\begin{cases}
			A_{ij,kh}, &i,k \leq l ~and ~j,h > l, \\
			0, &otherwise,
		\end{cases}
	\end{align*}
	for $i\neq j, k\neq h$. Together with \eqref{l2-dimA}, this equivalence can be shown to imply \eqref{AinA} by a straightforward verification, so that $A \in \mathscr{A}$.
\end{proof}

\begin{rem}\label{R(Q)=R(Qc)}
	It follows directly from \eqref{CONSTofQ} that $\mathrm{rank}(Q) = \mathrm{rank}(Q_c)$. Moreover, $Q_c$ has at least $l$ zero rows and $l$ zero columns,
	as the $i,(i+l)$-th rows and $k,(k+l)$-th columns of $Q$ are entirely zero for any $1\leq i,k\leq l$ by \eqref{Q_ihatikh}.
\end{rem}

In the following, we always assume that $Q$ satisfies \eqref{CONSTofQ} of Lemma \ref{AP-psd}.

Before proving Lemma \ref{BR-psd}, we introduce the following notation. Let $E_0:=I_l$. Let $\{v_q\}_{q=1}^l \subset \mathbb{R}^l$ and $\{w_{\alpha}\}_{\alpha=1}^m \subset \mathbb{R}^m$ be the standard basis row vectors, meaning the $q$-th component of $v_q$ and the $\alpha$-th component of $w_{\alpha}$ are $1$, with all other components being $0$. For each $q$ with $1 \leq q \leq l$, we form a matrix $R_q \in M(m \times l,~\R)$ by taking the $q$-th row of each matrix $E_0, \cdots, E_{m-1}$ (see \eqref{Cliffordsys-alg}), arranging them in order as row vectors, and combining them into a new matrix, i.e., 
\begin{equation}\label{Define R_q}
	R_q:=\begin{pmatrix}
		v_q E_0 \\
		\vdots \\
		v_q E_{m-1}
	\end{pmatrix}, \quad 1 \leq q \leq l.
\end{equation}
Define
\begin{equation}\label{Define R}
	R:=(R_1, \cdots, R_l) \in M(m \times l^2,~\R).
\end{equation} 
For each $1\leq \alpha\leq m$ and $1\leq i,j\leq l$, let $E_{\alpha-1}=(E^{\alpha-1}_{ij})_{l\times l}$ and $R=(R_{\alpha,ij})_{m\times l^2}$ where $ R_{\alpha,ij} $ denotes the entry of $ R $ at the $ \alpha $-th row and $ ((i-1)l + j) $-th column. Then we have $ R_{\alpha,ij}=E^{\alpha-1}_{ij} $. Hence 
$$R^{T}R=\(\sum_{\alpha=1}^m E^{\alpha-1}_{ij}E^{\alpha-1}_{kh}\)_{l^2 \times l^2}=\(\sum_{\alpha=1}^m P^{\alpha}_{i(j+l)}P^{\alpha}_{k(h+l)}\)_{l^2 \times l^2}.$$
Note that throughout the paper the indices $ij$ and $kh$ follow the lexicographic order
(i.e., $\{11,12,\dots,1l,\dots,l1,\dots,ll\}$).

From now on, let $1\leq i,j,k,h \leq l$, we denote $B=\(b_{ij,kh}\)_{l^2 \times l^2}:=\frac{1}{4}\(A_{i(j+l),k(h+l)}\)_{l^2 \times l^2}$. 
If $A\in \mathscr{A}$ and $Q$ is positive semidefinite, by \eqref{l2-dimA}, then the entries of $B$ satisfy
    \begin{empheq}[left=\empheqlbrace]{align}
		b_{ij,ij}&=1, \quad b_{ij,ih}=0  ~(h \neq j), \label{b_ijij}\\
		b_{ii,kk}&=1, \quad  b_{ii,kh}=0  ~(h \neq k), \label{b_iikk}\\
		b_{ij,kh}&=b_{kh,ij}, \label{b_sym}\\
		b_{ij,kh}&=-b_{ih,kj}  ~(k \neq i), \label{b_anti}\\
		b_{ij,kh}&=-b_{ji,kh} ~(j \neq i). \label{b_ji}
	\end{empheq}

Using the notations defined above, we can rewrite Lemma \ref{AP-psd} as:
\begin{lem}\label{BR-psd}
	$A\in \mathscr{A}$ and $Q$ is positive semidefinite if and only if the matrix $B$ satisfies \eqref{b_ijij}--\eqref{b_ji} and the matrix 
	$(B-R^{T}R)$
	is positive semidefinite.
\end{lem}
\begin{proof}
	It is readily verified that \eqref{l2-dimA} is equivalent to \eqref{b_ijij}--\eqref{b_ji}.
	And, as previously assumed, \eqref{CONSTofQ} always holds. 
	Therefore, the conclusion follows immediately from Lemma \ref{AP-psd} by noting that 
	\begin{equation}\label{B-RTR}
		B-R^{T}R=\dfrac{1}{4}\(A_{i(j+l),k(h+l)}-4\sum_{\alpha=1}^m P^{\alpha}_{i(j+l)}P^{\alpha}_{k(h+l)}\)_{l^2 \times l^2}=\dfrac{1}{4}Q_c.
	\end{equation} 
\end{proof}

Let 
\begin{equation}\label{widetilde-X}
	\widetilde{X}:=\bigl(x_i x_{l+j}\bigr)_{1\le i,j\le l}\in\mathbb R^{l^2}
\end{equation}
ordered lexicographically by $(i,j)$.
Then, directly from Remark \ref{R(Q)=R(Qc)} and \eqref{B-RTR}, we obtain:
\begin{prop}\label{R(Q)=R(B-RTR)}
	\[
	\mathrm{rank}(Q)=\mathrm{rank}\bigl(B-R^{T}R\bigr),
	\]
	and moreover
	\[
	G_F(x)=X^{T}QX=4\,\widetilde{X}^{T}\bigl(B-R^{T}R\bigr)\widetilde{X}.
	\]
\end{prop}

 Let the matrix $B$ be partitioned into $l \times l$ blocks $\(B_{ik}\)_{i,k=1}^{l}$, where each $B_{ik}$ is an $l \times l$ matrix whose $(j,h)$-entry is given by 
 \begin{equation}\label{Bik_jh=bijkh}
 	(B_{ik})_{jh} = b_{ij,kh}.
 \end{equation}
\begin{lem}\label{BRthm}
	The matrix 
	$B-R^{T}R$
	is positive semidefinite if and only if $B$ is positive semidefinite and $R_{i}B_{ij}=R_{j}$ for $1 \leq i,j \leq l$. 
\end{lem}
\begin{proof}
	We first note that
	\[
	B-R^TR\succeq0
	\quad\Longleftrightarrow\quad
	\mathcal{B}:=\begin{pmatrix} I_m & R\\ R^T & B\end{pmatrix}\succeq0,
	\]
	since $I_m$ is positive definite and $B-R^TR$ is the Schur complement of $I_m$ in $\mathcal{B}$.
	
	By \eqref{Define R_q}, we have
	\[
	R_qR_q^T=\bigl(v_qE_{\alpha-1}E_{\beta-1}^Tv_q^T\bigr)_{\alpha,\beta=1}^m.
	\]
	Now $E_0=I_l$, each $E_{\alpha-1}$ is orthogonal, and for $\alpha\neq\beta$ the matrix
	$E_{\alpha-1}E_{\beta-1}^T$ is skew-symmetric by the Clifford relations. Hence
	\[
	v_qE_{\alpha-1}E_{\beta-1}^Tv_q^T=
	\begin{cases}
		1,& \alpha=\beta,\\
		0,& \alpha\neq\beta,
	\end{cases}
	\]
	and therefore
	\begin{equation}\label{R_qR^T_q}
		R_qR_q^T=I_m
	\end{equation}
	for $1\le q\le l$.
	Since $B = (B_{ik})_{i,k=1}^{l}$ and $R = (R_1, \cdots, R_l)$, we can view $\mathcal{B}$ as an $(l+1) \times (l+1)$ block matrix and perform elementary row and column operations on it to annihilate the upper-left identity submatrix while preserving the lower-right block $B$. Specifically, for any $1\leq i\leq l$,
	\begin{itemize}
		\item left-multiply the $(i+1)$-th row by $-R_i$ and add it to the first row;
		\item right-multiply the $(i+1)$-th column by $-R_i^{T}$ and add it to the first column.
	\end{itemize}
	Hence we get 
	\begin{equation}\label{RBij}
		\begin{pmatrix}
			0 &  (R_j-R_iB_{ij})_{j=1}^l \\
			{(R_j-R_iB_{ij})_{j=1}^l} ^{T} & B
		\end{pmatrix},
	\end{equation}
	where $(R_j-R_iB_{ij})_{j=1}^l$ is an $1\times l$ block matrix with its $j$-th block being the $m \times l$ matrix $(R_j-R_iB_{ij})$.
	The block matrix \eqref{RBij} is positive semidefinite if and only if $B$ is positive semidefinite and $R_{i}B_{ij}=R_{j}$ for $1 \leq i,j \leq l$.
\end{proof}

Combining Lemmas \ref{SDP}, \ref{BR-psd} and \ref{BRthm}, one easily obtains:

\begin{prop}\label{QBcor}
	The \emph{psd} form $G_F$ in \eqref{nonnegativepolyG} on $\mathbb{R}^{n}$ is \emph{sos} if and only if there exists an $l^2 \times l^2$ matrix $B$ satisfying
	\begin{enumerate}
		\item conditions \eqref{b_ijij}--\eqref{b_ji}; \label{Bentry}
		\item $R_{i}B_{ij}=R_{j}$ for $1 \leq i,j \leq l$; \label{BRrelation}
		\item $B$ is positive semidefinite. \label{Bpsd}
	\end{enumerate}
\end{prop}
Proposition~\ref{QBcor} gives a preliminary SDP characterization of the \emph{sos} property of $G_F$ in terms of the matrix $B$. To prove the more concise characterization in Theorem~\ref{thm:SDP-feasible-B}, we next analyze the structural properties of matrices $B$ satisfying the conditions in Proposition~\ref{QBcor}. These properties will allow us to simplify the constraints in Proposition~\ref{QBcor} and thereby complete the proof of Theorem~\ref{thm:SDP-feasible-B}. They will also be used later in the analysis of the \emph{sos} and non-\emph{sos} cases.

\begin{lem}\label{properties of B}
	If $B=(B_{ik})_{i,k=1}^{l}$ satisfies \eqref{b_ijij}--\eqref{b_ji}, then $B_{ii}=I_l$,
	the matrix $B$ is symmetric, each off-diagonal block $B_{ik}$ is skew-symmetric, and
	$B_{ki}=-B_{ik}$ for $i\neq k$.
\end{lem}
\begin{proof}
	Note that $(B_{ik})_{jh} = b_{ij,kh}$ by \eqref{Bik_jh=bijkh}. Thus $B_{ii}=I_l$ by \eqref{b_ijij}, $B$ is symmetric by \eqref{b_sym} and  $B_{ik}$ ($i \neq k$) is skew-symmetric  by \eqref{b_anti}, which implies 
	$$B_{ki}=B_{ik}^T=-B_{ik}\ (i \neq k).$$
\end{proof}

Lemma~\ref{properties of B} shows that the conditions \eqref{b_ijij}--\eqref{b_ji} impose a rigid block structure on $B$: the diagonal blocks are identity matrices, while the off-diagonal blocks are skew-symmetric. We next introduce the involutions $\tau_k$, which provide a convenient way to describe certain special skew-symmetric blocks that will arise from the relations $R_iB_{ij}=R_j$. The following lemma makes this connection precise.

For $s\in\mathbb{N}^+$ and $1\le k\le s$, let
\[
I_s^{(k)}:=\mathrm{diag}(1,\ldots,1,\!-1,1,\ldots,1)\in M(s,\mathbb{R}),
\]
where the entry $-1$ appears in the $k$-th diagonal position.
For each fixed $k\in\mathbb{N}^+$ and each $s\ge k$, we define a map (still denoted by $\tau_k$)
\begin{equation}\label{tau_{k}}
	\tau_k: M(s,\mathbb{R})\longrightarrow M(s,\mathbb{R}),\qquad
	\tau_k(E):=I_s^{(k)}\,E\,I_s^{(k)}.
\end{equation}
Equivalently, $\tau_k$ multiplies both the $k$-th row and the $k$-th column of $E$ by $-1$.

Recall that $\{v_q\}_{q=1}^l \subset \mathbb{R}^l$ and $\{w_{\alpha}\}_{\alpha=1}^m \subset \mathbb{R}^m$ are the standard basis row vectors, defined such that the $q$-th component of $v_q$ and the $\alpha$-th component of $w_{\alpha}$ equal $1$, while all other components are $0$.

\begin{lem}\label{Btauthm}
	Assume that $B=\(B_{ik}\)_{i,k=1}^{l}$ satisfies \eqref{b_ijij}--\eqref{b_ji}. Given $1 \leq i \neq j \leq l$ and $2 \leq \alpha \leq m$, if 
	\begin{equation}\label{v_j B_{ik}}
		v_j B_{ik}=\pm w_{\alpha} R_k,\quad \forall 1 \leq k \leq l,
	\end{equation}
	then $B_{ij}=\mp \tau_{i}(E_{\alpha-1})=\mp \tau_{j}(E_{\alpha-1})$.
\end{lem}
\begin{proof}
	For $B_{ij}=\(b_{ik,jh}\)_{k,h=1}^l$, when $k \notin \{i, j\}$, 
	\begin{equation}\label{bikjh=-bijkh}
		b_{ik,jh}\overset{\eqref{b_ji}}{=}-b_{ki,jh}\overset{\eqref{b_anti}}{=}b_{kh,ji}\overset{\eqref{b_sym}}{=}b_{ji,kh}\overset{\eqref{b_ji}}{=}-b_{ij,kh}, ~\forall  1 \leq h \leq l.
	\end{equation}
	Hence, by the assumption, we have
	$$v_k B_{ij}\overset{\eqref{bikjh=-bijkh}}{=}-v_j B_{ik}=\mp w_{\alpha} R_k\overset{\eqref{Define R_q}}{=}\mp v_k E_{\alpha-1},\ k \notin \{i, j\}.$$
	
	When $k \in \{i, j\}$, $b_{ik,jh}=b_{ij,kh}$ for all $1 \leq h \leq l$ since
	\begin{equation}\label{biijh=bijih}
		b_{ii,jh}\overset{\eqref{b_iikk}}{=}\delta_{jh}\overset{\eqref{b_ijij}}{=}b_{ij,ih}.
	\end{equation}
	By the assumption, we have
	$$v_k B_{ij}=v_j B_{ik}=\pm w_{\alpha} R_k=\pm v_k E_{\alpha-1},\ k \in \{i, j\}.$$
	
	Combining the above two cases, $B_{ij}=\mp I_{l}^{(i)}I_{l}^{(j)}E_{\alpha -1}$. 
	By \eqref{biijh=bijih} and 
	$$b_{ij,jh}\overset{\eqref{b_ji}}{=}-b_{ji,jh}\overset{\eqref{b_ijij}}{=}-\delta_{ih},\ \forall 1\leq h\leq l,$$
	the matrix $B_{ij}$ has the $i$-th row equal to $v_j$ and the $j$-th row equal to $-v_i$. Since $B_{ij}$ is skew-symmetric (see Lemma \ref{properties of B}), it follows that the $i$-th column is $-v_j^T$ and the $j$-th column is $v_i^T$.
	Therefore,
	\begin{align*}
		&B_{ij}=I_{l}^{(j)}B_{ij}I_{l}^{(i)}=\mp I_{l}^{(i)}E_{\alpha -1}I_{l}^{(i)}=\mp \tau_{i}(E_{\alpha-1}),\\
		&B_{ij}=I_{l}^{(i)}B_{ij}I_{l}^{(j)}=\mp I_{l}^{(j)}E_{\alpha -1}I_{l}^{(j)}=\mp \tau_{j}(E_{\alpha-1}).
	\end{align*}	
\end{proof}

Lemma~\ref{Btauthm} identifies the precise form of some off-diagonal blocks once their interaction with the matrices $R_k$ is prescribed. We next record a general consequence of positive semidefiniteness, showing that an orthogonal off-diagonal block forces a multiplicative relation among the other blocks of $B$. This observation will be used repeatedly in the sequel.

\begin{lem}\label{BBthm}
	For some fixed $1\leq i,j\leq l$, if $B=\(B_{ik}\)_{i,k=1}^{l}$ is a positive semidefinite matrix satisfying \eqref{b_ijij} and the block $B_{ij}$ is an orthogonal matrix, then 
	$$B_{ik} = B_{ij}B_{jk}$$ 
	for all $1 \leq k \leq l$.
\end{lem}
\begin{proof}
	Under the given conditions, we have  $B_{ji}B_{ij}=B^{T}_{ij}B_{ij}=I_{l}$ and $B_{hh}=I_{l}$ for all $1\leq h\leq l$.
	
	When $i=j$, the conclusion holds trivially. Now consider the case $i \neq j$. For $k = i$ or $k = j$, the relation $B_{ik} = B_{ij}B_{jk}$ follows directly. For any $k \notin \{i,j\}$, 
	consider the principal submatrix of $B$ corresponding to the $i$-, $j$-, and $k$-th block rows and columns
	(permute block indices so that the order is $\{i,j,k\}$). Applying a congruence transformation yields
	\begin{equation*}
		\begin{pmatrix}
			I_l &  B_{ij} & B_{ik} \\
			B_{ji} & I_l & B_{jk} \\
			B_{ki} & B_{kj} &I_l
		\end{pmatrix}
		\rightarrow 
		\begin{pmatrix}
			0 &  0 & B_{ik}-B_{ij}B_{jk} \\
			0 & I_l & B_{jk} \\
			B_{ki}-B_{kj}B_{ji} & B_{kj} &I_l
		\end{pmatrix}.
	\end{equation*}
	Since the principal submatrix of $B$ is positive semidefinite, it follows that $B_{ik} = B_{ij}B_{jk}$.
\end{proof}

We now complete the proof of Theorem~\ref{thm:SDP-feasible-B}. 
By Proposition~\ref{QBcor}, it suffices to show that the SDP constraints in Theorem~\ref{thm:SDP-feasible-B} imply all the relations \eqref{b_ijij}--\eqref{b_ji}. 
Among these, Lemma~\ref{properties of B} already yields \eqref{b_ijij}, \eqref{b_sym}, and \eqref{b_anti}: indeed, from the block formulation in Theorem~\ref{thm:SDP-feasible-B} we know that $B_{ii}=I_l$, that $B$ is symmetric, and that each off-diagonal block $B_{ij}$ is skew-symmetric. 
Therefore the only relations from Proposition~\ref{QBcor} that still need to be recovered are \eqref{b_iikk} and \eqref{b_ji}. We verify them below.

Recall that the first row of $R_i$ is the $i$-th row of $E_0=I_l$, namely $v_i$.
Taking the first row on both sides of $R_iB_{ik}=R_k$ yields
\[
v_i B_{ik}=v_k \qquad (1\le i,k\le l).
\]
Using $(B_{ik})_{jh}=b_{ij,kh}$, we obtain
\[
b_{ii,kh}=\delta_{kh},
\]
which is exactly \eqref{b_iikk}.

Fix $i\neq j$ and arbitrary $k,h$. We verify \eqref{b_ji}. 
If $(k,h)=(i,j)$ or $(k,h)=(j,i)$, then the conclusion follows directly from
\eqref{b_ijij}, \eqref{b_iikk}, and \eqref{b_anti}. 

Now assume $(k,h)\notin\{(i,j),(j,i)\}$. Since $B\succeq0$, the $3\times3$ principal
submatrix of $B$ indexed by the three distinct index pairs $ij$, $ji$, and $kh$ is
positive semidefinite. Using $B_{ii}=I_l$ and the skew-symmetry of $B_{ij}$, its
determinant reduces to
\[
-(b_{ij,kh}+b_{ji,kh})^{2}\ge0,
\]
and hence $b_{ij,kh}+b_{ji,kh}=0$. Therefore \eqref{b_ji} holds for all $k,h$.

This verifies the remaining relations required in Proposition~\ref{QBcor}, and hence completes the proof of Theorem~\ref{thm:SDP-feasible-B}.

\section{A Reduction to Representative Cases for Theorem~\ref{sos thm}}\label{Simplify}
In this section we use the representation theory of irreducible Clifford systems to derive a reduction
principle for \emph{sos} certification (see Proposition~\ref{reduction principle}).  This principle
substantially decreases the number of multiplicity pairs that need to be checked individually.

Recall from \cite{FKM81} that every Clifford system is algebraically equivalent to a direct sum of irreducible Clifford systems. Let $\delta(m)$ denote the minimal dimension of an irreducible real representation of the Clifford algebra $C_{m-1}$.
Then an irreducible Clifford system $\{P_0,\cdots,P_m\}$ on $\mathbb{R}^{2l}$ exists precisely for the following values of $m$ with $l=\delta(m)$:
	\begin{table}[htbp]
		\centering
		\begin{tabular}{|c|c|c|c|c|c|c|c|c|c|}
			\hline
			$m$ & $1$ & $2$ & $3$ & $4$ & $5$ & $6$ & $7$ & $8$ & $\cdots~ m+8$\\
			\hline
			$\delta(m)$& $1$ &$2$ &$4$ &$4$ & $8$ & $8$ & $8$ & $8$ & $\cdots~ 16\delta(m)$\\
			\hline
		\end{tabular}
	\vspace{1em}
	\caption{The minimal dimension $\delta(m)$ of an irreducible real representation of the Clifford algebra $C_{m-1}$}
	\label{representation of Cl}
	\vspace{-1em}
\end{table}

Consider the decomposition of $\{P_0,\cdots,P_m\}$ on $\mathbb{R}^{2l}$ with $l=k\delta(m)$ into a direct sum of $k\geq 1$ irreducible Clifford systems on $\mathbb{R}^{2\delta(m)}$ (denoted with a superscript $r=1,\cdots,k$) so that
\begin{equation}\label{Clifforddecom-irr}
	\begin{array}{cccc}
		\mathbb{R}^{2l}=&\mathbb{R}^{2\delta(m)} & \oplus  \cdots \oplus & \mathbb{R}^{2\delta(m)}  \\
		(P_0,\cdots,P_m)=&(P_0^1,\cdots,P_m^1) & \oplus \cdots  \oplus & (P_0^k,\cdots,P_m^k).
	\end{array}
\end{equation}
Here the irreducible Clifford systems $\{P_0^r,\cdots,P_m^r\}$ on $\mathbb{R}^{2\delta(m)}$ can be expressed in the form as (\ref{Cliffordsys-alg}) so that
\begin{equation}\label{Cliffordsys-alg-irr}
	P_0^r=\begin{pmatrix}
		I_{\delta(m)} &  0 \\
		0 & -I_{\delta(m)}
	\end{pmatrix}, \quad P_1^r=\begin{pmatrix}
		0 & I_{\delta(m)}     \\
		I_{\delta(m)}  & 0
	\end{pmatrix}, \quad P_{\alpha+1}^r=\begin{pmatrix}
		0 & E_\alpha^r    \\
		-E_\alpha^r & 0
	\end{pmatrix},  \\
	~~
\end{equation}
$\alpha=1,\cdots,m-1,$
where $\{E_1^r,\cdots,E_{m-1}^r\}$ generates an irreducible Clifford algebra on each $\mathbb{R}^{\delta(m)}$ of the decomposition of $\{E_1,\cdots,E_{m-1}\}$ on $\mathbb{R}^{l}=\mathbb{R}^{\delta(m)}  \oplus  \cdots \oplus  \mathbb{R}^{\delta(m)}$.
The multiplicities of an isoparametric hypersurface of OT-FKM type are
\begin{equation*}
	m_+ = m, \quad m_- = l-m-1=k\delta(m) - m - 1, \quad k \geq 1,
\end{equation*}
where $k$ is chosen sufficiently large so that $m_->0$. In the table below of possible multiplicities of the principal curvatures of an isoparametric hypersurface of OT-FKM type, the cases where $m_- \leq 0$ are denoted by a dash.

\begin{table}[htbp]
	\centering
	\begin{tabular}{|c|c|c|c|c|c|c|c|c|c|c|}
		\hline
		\diagbox[width=3.5em, height=2.7em, innerleftsep=10pt, innerrightsep=0pt]{$k$}{$\delta(m)$} & $1$ & $2$ & $4$ & $4$ & $8$ & $8$ & $8$ & $8$ & $16$ & $\cdots$ \\
		\hline
		$1$ & $-$ & $-$ & $-$ & $-$ & $(5, 2)$ & $(6, 1)$ & $-$ & $-$ & $(9, 6)$ & $\cdots$ \\[0.3em]
		\hline
		$2$ & $-$ & $(2, 1)$ & $(3, 4)$ & \underline{$(4, 3)$} & $(5, 10)$ & $(6, 9)$ & $(7, 8)$ & \underline{$(8, 7)$} & $(9, 22)$ & $\cdots$\\[0.3em]
		\hline
		$3$ & $(1, 1)$ & $(2, 3)$ & $(3, 8)$ & \underline{$(4, 7)$} & $(5, 18)$ & $(6, 17)$ & $(7, 16)$ & \underline{$(8, 15)$} & $(9, 38)$ & $\cdots$\\[0.3em]
		\hline
		$4$ & $(1, 2)$ & $(2, 5)$ & $(3, 12)$ & \underline{\underline{$(4, 11)$}} & $(5, 26)$ & $(6, 25)$ & $(7, 24)$ & \underline{\underline{$(8, 23)$}} & $(9, 54)$ & $\cdots$\\[0.3em]
		\hline
		$5$ & $(1, 3)$ & $(2, 7)$ & $(3, 16)$ & \underline{\underline{$(4, 15)$}} & $(5, 34)$ & $(6, 33)$ & $(7, 32)$ & \underline{\underline{$(8, 31)$}} & $(9, 70)$ & $\cdots$\\[0.3em]
		\hline
		$\vdots$ & $\vdots$ & $\vdots$ & $\vdots$ & $\vdots$ & $\vdots$ & $\vdots$ & $\vdots$ & $\vdots$ & $\vdots$ & $\ddots$\\
		\hline
	\end{tabular}
	\vspace{1em}
	\caption{Multiplicities of principal curvatures of OT-FKM type hypersurfaces}\label{table of m+,m-}
	\vspace{-1em}
\end{table}

Geometrically equivalent Clifford systems determine congruent families of isoparametric hypersurfaces. In Table \ref{table of m+,m-}, the underlined multiplicities,
\[
\underline{(m_+, m_-)}, \quad \underline{\underline{(m_+, m_-)}},
\]
denote the two, respectively, three geometrically inequivalent Clifford systems for the multiplicities $(m_+, m_-)$. Ferus, Karcher, and Münzner show that these geometrically inequivalent Clifford systems with $m \equiv 0 \pmod{4}$ and $l = k\delta(m)$ actually lead to incongruent families of isoparametric hypersurfaces, of which there are $\lfloor k/2\rfloor+1$.

\begin{lem}\label{geom. equi. with sos}
	The \emph{sos} property of $G_F$ is invariant under geometric equivalence of Clifford systems; that is, if $G_F$ is \emph{sos} for one Clifford system in an equivalence class, then it is \emph{sos} for all Clifford systems in that class.
\end{lem}

\begin{proof}
	Assume $\{P_{0},\cdots,P_{m}\}$ and $\{{P}_{0}' ,\cdots,{P}_{m}'\}$ are two geometrically equivalent Clifford systems on $\mathbb{R}^{2l}$, and denote
	$$G_F(x):=|x|^4-\sum_{\alpha=0}^{m}\langle P_\alpha x,x\rangle^2,\quad {G}_F'(x):=|x|^4-\sum_{\alpha=0}^{m}\langle {P}_\alpha' x,x\rangle^2.$$
	It suffices to prove that if $G_F$ is \emph{sos}, then ${G}_F'$ is \emph{sos}.
	
	Suppose that $G_F$ is \emph{sos}. Since $\{P_{0},\cdots,P_{m}\}$ and $\{{P}_{0}' ,\cdots,{P}_{m}'\}$ are geometrically equivalent, there exist an orthogonal transformation $U\in O(\mathrm{Span}\{P_0, \cdots,P_m\})$ and an orthogonal matrix $W\in O(\mathbb{R}^{2l})$   such that
	$$
	P_{\alpha}'= W^T U(P_{\alpha}) W,\quad \forall \alpha=0,1,\dots,m.
	$$
	Then there exists $\(u_{\alpha}^{\beta}\)_{\alpha,\beta=0}^m \in O(m+1)$ such that
	$$U(P_{\alpha})=\sum_{\beta=0}^{m} u_{\alpha}^{\beta} P_{\beta},\quad \forall \alpha=0,1,\dots,m.$$
	Thus we have
	\begin{align*}
		{G}_F'(x)&=|x|^4-\sum_{\alpha=0}^{m}\langle W^T U(P_{\alpha}) W x,x\rangle^2\\
		&=|Wx|^4-\sum_{\alpha=0}^{m}\langle  U(P_{\alpha}) W x,Wx\rangle^2\\
		&=|Wx|^4-\sum_{\alpha=0}^{m}\sum_{\beta,\gamma=0}^{m} u_{\alpha}^{\beta} u_{\alpha}^{\gamma}\langle  P_{\beta} W x,Wx\rangle \langle  P_{\gamma} W x,Wx\rangle\\
		&=|Wx|^4-\sum_{\beta,\gamma=0}^{m} \(\sum_{\alpha=0}^{m} u_{\alpha}^{\beta} u_{\alpha}^{\gamma}\)\langle  P_{\beta} W x,Wx\rangle \langle  P_{\gamma} W x,Wx\rangle\\
		&=|Wx|^4-\sum_{\beta,\gamma=0}^{m} \delta_{\beta\gamma}\langle  P_{\beta} W x,Wx\rangle \langle  P_{\gamma} W x,Wx\rangle\\
		&=|Wx|^4-\sum_{\beta=0}^{m} \langle  P_{\beta} W x,Wx\rangle^2={G}_F(Wx).
	\end{align*}
	This implies that ${G}_F'(x)$ is \emph{sos}.
\end{proof}

Note that henceforth, when we say $G_F$ is \emph{sos} for the pair $(m,l)$, we mean that for any Clifford system $\{P_0,\cdots,P_m\}$ on $\mathbb{R}^{2l}$, the polynomial $G_F(x) = |x|^4 - \sum_{\alpha=0}^m \langle P_\alpha x,x \rangle^2$ is \emph{sos}.

From the lemma above, we obtain the main proposition of this section:

\begin{prop}\label{reduction principle}
	If $G_F$ is \emph{sos} for $(m,l) = (m_0, l_0)$, then $G_F$ is \emph{sos} for all pairs $(m_1, l_0)$ with $1\leq m_1 \leq m_0$ and $m_1 \not\equiv 0 \pmod{4}$.
\end{prop}
\begin{proof}
	Assume $\{P_0,\cdots,P_{m_0}\}$ is a Clifford system on $\mathbb{R}^{2l_0}$. Then for any $m_1$ satisfying $1\leq m_1 \leq m_0$,  $\{P_0, \cdots, P_{m_1}\}$ is also a Clifford system on $\mathbb{R}^{2l_0}$. Denote  
	$$
	G_F^0(x) := |x|^4 - \sum_{\alpha=0}^{m_0} \langle P_\alpha x, x \rangle^2, \quad G_F^1(x) := |x|^4 - \sum_{\alpha=0}^{m_1} \langle P_\alpha x, x \rangle^2.
	$$ 
	Since $ G_F^0(x) $ is \emph{sos}, and observe that  
	$$
	G_F^1(x)=|x|^4 - \sum_{\alpha=0}^{m_0} \langle P_\alpha x, x \rangle^2 + \sum_{\alpha=m_1+1}^{m_0} \langle P_\alpha x, x \rangle^2 = G_F^0(x) + \sum_{\alpha=m_1+1}^{m_0} \langle P_\alpha x, x \rangle^2,
	$$
	it follows that $ G_F^1(x) $ is also \emph{sos}.  
	
	For $m_1 \not\equiv 0 \pmod{4}$, there exists exactly one geometric equivalence class of Clifford systems on $\mathbb{R}^{2l_0}$ (see \cite{Cecil15}). Then, by Lemma \ref{geom. equi. with sos}, the fact that $G_F^1(x)$ is \emph{sos} implies that $G_F$ is \emph{sos} for $(m,l)=(m_1, l_0)$.
\end{proof}

This proposition reduces the problem to proving the \emph{sos} and \emph{non-sos} property of $G_F$ for some multiplicity pairs $(m_+,m_-) = (m, l-m-1)$ listed in Theorem \ref{sos thm}.

\begin{cor}\label{reduced cases}
		To prove Theorem \ref{sos thm}, it suffices to verify the following:
		\begin{enumerate}
			\item $G_F$ is \emph{non-sos} for: 
			\begin{enumerate}
				\item  $(m_+,m_-) =(m, l-m-1)= (4,3)^D$ (of definite class),
				\item  $(m,l)=(3,4r)$ for all $r\geq 3$;
			\end{enumerate}
			\item $G_F$ is \emph{sos} for:
			\begin{enumerate}
				\item $(m,l) = (1,k+2)$ for all $k \in \mathbb{N}^+$,
				\item $(m,l) = (2,2k+2)$ for all $k \in \mathbb{N}^+$,
				\item $(m,l) = (6,8)$.
			\end{enumerate}
		\end{enumerate}
\end{cor}

\begin{proof}
	We claim that $G_F$ is \emph{non-sos} for all pairs $(m,l)$ with $m\geq 3$ and $l=k\delta(m)\geq 12$. Indeed, suppose for contradiction that $G_F$ were \emph{sos} for some such pair $(m_0,l_0)$. Then by Proposition \ref{reduction principle}, $G_F$ would also be \emph{sos} for $m=3$ and $l=l_0$ (where $l_0\geq 12$), contradicting condition~(1)(b).
	
	Assume that $\{P_0,\cdots,P_{6}\}$ is a Clifford system on $\mathbb{R}^{16}$. By condition~(2)(c), the polynomial $G_F$ associated with $\{P_0,\cdots,P_{6}\}$ is \emph{sos}. Therefore, by the proof of Proposition \ref{reduction principle}, the polynomial $G_F$ associated with $\{P_0,\cdots,P_{4}\}$ is also \emph{sos}. When $m=4$, there are two geometric equivalence classes of Clifford systems on $\mathbb{R}^{16}$, namely, the definite class and the indefinite class. The system $\{P_0,\cdots,P_{4}\}$ must belong to the indefinite class, because the polynomial $G_F$ in the definite class is \emph{non-sos} by condition~(1)(a). Consequently, together with Lemma \ref{geom. equi. with sos}, this shows that $G_F$ is \emph{sos} for $(m_+,m_-)=(4,3)^I$.
	
	Applying Proposition \ref{reduction principle} once more, condition~(2)(c) implies that $G_F$ is \emph{sos} for $(m,l)=(3,8)$ and $(5,8)$. At this stage, we have established the \emph{sos} or \emph{non-sos} property of $G_F$ for all multiplicity pairs listed in Table~\ref{table of m+,m-}, thus completing the proof of Theorem \ref{sos thm}.
\end{proof}

\section{The \emph{Non-sos} Cases in Theorem~\ref{sos thm}}\label{pr-nonsos}

In this section, we prove the \emph{non-sos} cases in Theorem~\ref{sos thm}. 
By Lemma~\ref{geom. equi. with sos}, it suffices to verify the \emph{sos} property of $G_F$ for a single representative Clifford system in each geometric equivalence class. 
Accordingly, by Corollary~\ref{reduced cases}, it remains to consider two types of multiplicities: the exceptional case $(m_+,m_-)=(4,3)^D$ and the family $(m,l)=(3,4r)$ with $r\geq 3$. 
For each case we choose a suitable representative Clifford system and show that the corresponding polynomial $G_F$ cannot be written as a sum of squares.

\subsection{The \emph{Non-sos} Case $(m_+,m_-)=(4,3)^D$}\label{non-sos (4,3)^D}
For $(m,l) = (m_+, m_++m_-+1)=(4,8)$, there are two geometric equivalence classes of Clifford systems on $\mathbb{R}^{2l}$, referred to as the indefinite class and the definite class (see \cite{Cecil15}). A Clifford system $\{P_0,\cdots,P_m\}$ is called definite if $P_0\cdots P_m=\pm I_{2l}$. In the case where $(m_+, m_-)=(4,3)^{D}$ with definite Clifford system $\{P_0,\cdots,P_m\}$, assuming the \emph{psd} form $G_F$ in \eqref{nonnegativepolyG} is \emph{sos}, we proceed with a proof by contradiction. 

Define a linear homomorphism $\iota: \CC \rightarrow M(2, \R)$ by
\begin{equation}\label{iota}
	\iota(1) := I_2, \quad \iota(\oi) := \begin{pmatrix} 0 & -1 \\ 1 & 0 \end{pmatrix}.
\end{equation}  
Further, for all $k \in \mathbb{N}^{+}$ and $E = (e_{ij})_{k\times k} \in M(k, \CC)$, define the linear homomorphism $\iota_k: M(k, \CC) \rightarrow M(2k, \R)$ by 
\begin{equation}\label{iota_k}
	\iota_k(E) := \(\iota(e_{ij})\)_{k\times k}.
\end{equation} 
Note that $\iota_1 = \iota$ and we call $\iota_k(E)$ the real matrix corresponding to $E$. 

A $2 \times 2$ complex matrix representation of Clifford algebra $C_{3}$ is given by
$$ -\oi\sigma_3, \quad \oi\sigma_2, \quad -\oi\sigma_1, $$
where 
\begin{equation}\label{Pauli}
	\sigma_1 = \begin{pmatrix} 
		0 & 1 \\ 1 & 0 \end{pmatrix},
	\quad \sigma_2 = \begin{pmatrix} 
		0 & -\oi \\ \oi & 0 \end{pmatrix},
	\quad \sigma_3 = \begin{pmatrix} 
		1 & 0 \\ 0 & -1 \end{pmatrix}
\end{equation}
are Pauli matrices.

Let $\widetilde{E}_1, \widetilde{E}_2, \widetilde{E}_3$ denote their corresponding real matrices, i.e., 
\begin{equation}\label{C3 on R4}
	\widetilde{E}_1 := -\iota_2(\oi\sigma_3), \quad \widetilde{E}_2 := \iota_2(\oi\sigma_2), \quad \widetilde{E}_3 := -\iota_2(\oi\sigma_1).
\end{equation}
We then construct an $8 \times 8$ real matrix representation of $C_{3}$ as follows:
$$E_1:=\widetilde{E}_1 \oplus \widetilde{E}_1, \quad E_2:= \widetilde{E}_2 \oplus \widetilde{E}_2, \quad E_3:=\widetilde{E}_3 \oplus \widetilde{E}_3. $$
Consider the Clifford system $\{P_0,\cdots,P_4\}$ on $\mathbb{R}^{16}$ obtained by substituting $E_1,E_2,E_3$ into  \eqref{Cliffordsys-alg}. A straightforward verification confirms that 
$$P_0 \cdots P_4 = (E_1 E_2 E_3)\ \oplus\ (E_1 E_2 E_3) = -I_{16},$$ thereby establishing that this is indeed the definite case. By the earlier assumption, the polynomial $G_F(x) = |x|^4 - \sum_{\alpha=0}^{4} \langle P_\alpha x, x \rangle^2$ is \emph{sos}.

Let $E_0=I_8$. For any $1\leq q\leq 8$ and $1\leq \alpha\leq 4$, the $\alpha$-th row of $R_q$ is the $q$-th row of $E_{\alpha-1}$ (see \eqref{Define R_q}).  One gets
$$ R_1 = \begin{pmatrix} I_4, O_4 \end{pmatrix}, \ R_2 = \begin{pmatrix} \tau_3(\widetilde{E}_1), O_4 \end{pmatrix}, \ R_3 = \begin{pmatrix} \tau_2(\widetilde{E}_2), O_4 \end{pmatrix}, \ R_4 = \begin{pmatrix} \tau_2(\widetilde{E}_3), O_4 \end{pmatrix}, $$
and $$ R_{i+4} = R_i \begin{pmatrix} O_4 & I_4 \\ -I_4 & O_4 \end{pmatrix}, \quad  i=1,2,3,4,$$ 
where $\tau_2$, $\tau_3$ are as defined in \eqref{tau_{k}} and $O_4$ denotes the $4\times 4$ zero matrix.

By Proposition \ref{QBcor}, there exists an $l^2\times l^2$ positive semidefinite matrix $B$ satisfying conditions \eqref{b_ijij}--\eqref{b_ji}, and $R_{i}B_{ij}=R_{j}$ for all $1 \leq i,j \leq l$. For all $1 \leq k \leq l$, we have
$ R_1 B_{1k} = R_k $, so that the first four rows of $B_{1k}$ equal $R_k$.
From the second row of $R_1 B_{1k} = R_k$, we have $v_2 B_{1k} = w_2 R_k$ for all $1 \leq k \leq l$, which implies that $B_{12} = -\tau_1(E_1)$ by Lemma \ref{Btauthm}. Consequently, by Lemma \ref{properties of B}, the matrix
$$B_{21} =-B_{12}= \tau_1(E_1)=\tau_1(\widetilde{E}_1) \oplus \widetilde{E}_1$$
is orthogonal.

Furthermore, considering the relations $R_1 B_{15} = R_5 \ \text{and} \ R_5 B_{15} =-R_5 B_{51} =-R_1,$
we conclude that 
$$
B_{15} = \begin{pmatrix} O_4 & I_4 \\ -I_4 & O_4 \end{pmatrix}.
$$
Then applying Lemma \ref{BBthm}, we obtain
$$B_{25} = B_{21}B_{15}=\begin{pmatrix}
	O_4 &  \tau_1(\widetilde{E}_1) \\
	-\widetilde{E}_1 & O_4 
\end{pmatrix},$$
which fails to be skew-symmetric. This yields a contradiction with Lemma \ref{properties of B}.

Therefore, $G_F$ is \emph{non-sos} in the case $(m_+, m_-)=(4,3)^{D}$.

\subsection{The \emph{Non-sos} Cases $(m,l)=(3,4r)$ with $r\ge 3$}\label{m=3 l=4r}
For the sake of contradiction, suppose the \emph{psd} form $G_F$ in \eqref{nonnegativepolyG} is \emph{sos} for $(m,l)=(3,4r)$. We still take the same matrices $\widetilde{E}_1, \widetilde{E}_2, \widetilde{E}_3$ as \eqref{C3 on R4}. Define the block-diagonal matrices 
\begin{equation}\label{C3 matries}
	E_0 := I_l, \quad E_1 :=\underbrace{\widetilde{E}_1 \oplus \cdots \oplus \widetilde{E}_1}_{r}, \quad E_2 :=\underbrace{\widetilde{E}_2 \oplus \cdots \oplus \widetilde{E}_2}_{r},
\end{equation}
then $\{E_1, E_2\}$ gives a real matrix representation of Clifford algebra $C_{2}$ on $\R^l$. Consider the Clifford system $\{P_0,\cdots,P_3\}$ on $\mathbb{R}^{2l}$ obtained by substituting $E_1,E_2$ into  \eqref{Cliffordsys-alg}. Then the polynomial $G_F(x) = |x|^4 - \sum_{\alpha=0}^{3} \langle P_\alpha x, x \rangle^2$ is \emph{sos}.

In the present case, for any $1\leq q\leq l$, $R_q$ is a $3 \times l$ matrix whose $\alpha$-th row (for $1 \leq \alpha \leq 3$) is given by the $q$-th row of $E_{\alpha-1}$, according to definition \eqref{Define R_q}.
Let 
\begin{equation}\label{Define Rj'}
	R_j' := \begin{pmatrix} 1 & 0 & 0 \\0 & 1 & 0 \\ 0 & 0 & 1 \\ 0 & 0 & 0 \end{pmatrix} R_j = \begin{pmatrix} R_j \\ O_{1\times l} \end{pmatrix}_{4 \times l},
\end{equation} 
where $O_{1\times l}$ denotes the $1\times l$ zero matrix.
For any $1 \leq i \leq 4$, let $D_i$ denote the matrix $I_4$ with the $i$-th row multiplied by $0$. For any $2 \leq s \leq r$, define the block matrix $J_s = (E_{1s}-E_{s1}) \otimes I_4$, where $\{E_{ij}: 1 \leq i,j \leq r\}$ denotes the set of $r \times r$ standard basis matrices, each having  $1$ in the $(i,j)$-entry and zeros elsewhere.
Then
$$ R_1' = \begin{pmatrix} D_4, O_4, \cdots, O_4 \end{pmatrix}, \quad R_2' = \begin{pmatrix} D_4 \tau_3(\widetilde{E}_1), O_4, \cdots, O_4 \end{pmatrix}, $$
$$ R_3' = \begin{pmatrix} D_4 \tau_2(\widetilde{E}_2), O_4, \cdots, O_4 \end{pmatrix}, \quad R_4' = \begin{pmatrix} D_4 \tau_2(\widetilde{E}_3), O_4, \cdots, O_4 \end{pmatrix}, $$
\begin{equation}\label{R_{4k+i}'}
	R_{4k+i}' = R_i' J_{k+1}, \quad 1\leq k\leq r-1, \quad 1\leq i\leq 4,
\end{equation}
where $\tau_2$, $\tau_3$ are as defined in \eqref{tau_{k}} and $O_4$ denotes the $4\times 4$ zero matrix.

By Proposition \ref{QBcor} and the defining equation \eqref{Define Rj'}, there exists $B$ satisfying \eqref{b_ijij}--\eqref{b_ji} such that $B$ is positive semidefinite and $R_{i}'B_{ij}=R_{j}'$ for all $1 \leq i,j \leq l$.

Recall that $B$ is partitioned into $l \times l$ blocks $(B_{ik})_{i,k=1}^{l}$, where each $B_{ik}$ is an $l \times l$ matrix. Let $(B_{ts}^{ik})_{1 \leq t,s \leq r}$ be the $r \times r$ block representation of $B_{ik}$, where $B_{ts}^{ik}$ is a $4 \times 4$ matrix. Since $B_{ik}$ is skew-symmetric, we have
$(B_{ts}^{ik})^T = -B_{st}^{ik}.$

First prove $B_{12}^{15} = I_{4}^{(4)}$, where $I_{4}^{(4)}$ is the same as in \eqref{tau_{k}}. Since $ R_1' B_{1k} = R_k'$ for all $1 \leq k \leq l$, we have
$$v_j B_{1k} = w_j R_k, ~j=1,2,3, ~\forall 1 \leq k \leq l, $$
where $v_j$ and $w_j$ are defined as in \eqref{v_j B_{ik}}.
By Lemma \ref{Btauthm},
\begin{equation}\label{B12 for (3,4r)}
	B_{12} = -\tau_1(E_1)=-(\underbrace{\tau_1(\widetilde{E}_1) \oplus \cdots \oplus \widetilde{E}_1}_{r})
\end{equation}
is orthogonal.

On one hand, the relations
$
R_1' B_{15} = R_5' \ \text{and} \ R_5' B_{15} = -R_5' B_{51}=-R_1'
$
lead to the conclusions that
\[
D_4 B_{12}^{15} = D_4, \ D_4 B_{21}^{15} = -D_4.
\]
Since $(B_{12}^{15})^T = -B_{21}^{15}$, we have
\begin{equation}\label{B1512}
	B^{15}_{12} =\odiag\{1,1,1,c\},
\end{equation}
where $c$ is to be determined. 

On the other hand, starting from the relation $R_2' B_{25} = R_5'$, we derive a sequence of implications. First, this implies $D_4 \tau_3(\widetilde{E}_1) B_{12}^{25} = D_4$. Multiplying both sides by $-\tau_3(\widetilde{E}_1)$ yields 
$$-\tau_3(\widetilde{E}_1) D_4 \tau_3(\widetilde{E}_1) B_{12}^{25} = -\tau_3(\widetilde{E}_1) D_4.$$ 
Moreover, since $\tau_3(\widetilde{E}_1) D_4 = D_3 \tau_3(\widetilde{E}_1)$ and since $\tau_3(\widetilde{E}_1)$ is a skew-symmetric orthogonal matrix, it follows that $$D_3 B_{12}^{25} = -D_3 \tau_3(\widetilde{E}_1).$$

And from the relation $R_5' B_{25} = -R_2'$, we directly obtain $$D_4 B_{21}^{25} = -D_4 \tau_3(\widetilde{E}_1).$$
Since $(B_{12}^{25})^T = -B_{21}^{25}$, we have
$$B_{12}^{25} =\begin{pmatrix} 0 & -1 & 0 & 0\\1 & 0 & 0 & 0\\ 0 & 0 & 0 & d \\ 0 & 0 & 1 & 0 \end{pmatrix},$$  where $d$ is to be determined.

Applying Lemma \ref{BBthm}, we derive the relation
$B_{15} = B_{12} B_{25}.$
By \eqref{B12 for (3,4r)}, we obtain
$$B_{21}^{15} = \sum_{s=1}^{r}B_{2s}^{12} B_{s1}^{25} =B_{22}^{12} B_{21}^{25} = \widetilde{E}_1 (B_{12}^{25})^T = \odiag\{-1,-1,-d,1\}. $$
It follows that
$$B_{12}^{15} = -(B_{21}^{15})^T = \odiag\{1,1,d,-1\},$$
which implies $d=1$ and $c=-1$ upon comparing with \eqref{B1512}. Consequently, $B_{12}^{15} = I_{4}^{(4)}$.

Observing \eqref{C3 matries} and \eqref{R_{4k+i}'}, $B_{ij}$ and $B_{kh}$ should have similar properties when $i \equiv k \pmod{4}$ and $j \equiv h \pmod{4}$. In fact, similarly to the above, we can compute that $B_{13}^{19} = I_{4}^{(4)}$, $B_{23}^{59} = I_{4}^{(4)}$.

Since $B$ is positive semidefinite, its principal submatrix
$$ S:=\begin{pmatrix} I_l & B_{15} & B_{19} \\ B_{51} & I_l & B_{59} \\ B_{91} & B_{95} & I_l \end{pmatrix} $$
must also be positive semidefinite.
Based on the preceding calculations, the matrix
$$ K:=\begin{pmatrix} I_4 & B_{12}^{15} & B_{13}^{19} \\ B_{21}^{51} & I_4 & B_{23}^{59} \\ B_{31}^{91} & B_{32}^{95} & I_4 \end{pmatrix} = \begin{pmatrix} I_4 & I_{4}^{(4)} & I_{4}^{(4)} \\ I_{4}^{(4)} & I_4 & I_{4}^{(4)} \\ I_{4}^{(4)} & I_{4}^{(4)} & I_4 \end{pmatrix}, $$
which is a principal submatrix of $S$.
$K$ must be positive semidefinite. However, we obtain a contradiction since
$$ \begin{pmatrix} I_4 & I_4 & I_4 \end{pmatrix} \begin{pmatrix} I_4 & I_{4}^{(4)} & I_{4}^{(4)} \\ I_{4}^{(4)} & I_4 & I_{4}^{(4)} \\ I_{4}^{(4)} & I_{4}^{(4)} & I_4 \end{pmatrix} \begin{pmatrix} I_4 \\ I_4 \\ I_4 \end{pmatrix} = 3I_4 + 6I_{4}^{(4)}=\odiag\{9,9,9,-3\} $$
contains negative values along its diagonal. 

Therefore, the polynomial $G_F(x) = |x|^4 - \sum_{\alpha=0}^{3} \langle P_\alpha x, x \rangle^2$ is \emph{non-sos}. For $m=3$, there exists exactly one geometric equivalence class of Clifford systems on $\mathbb{R}^{8r}$. By Lemma \ref{geom. equi. with sos}, $G_F$ is \emph{non-sos} for $(m,l)=(3,4r)$ with $r\geq 3$.

\begin{rem}
	In short, the reason why $G_F$ is \emph{non-sos} for $(m,l)=(3,4r)$ is that the matrix $B$ satisfying the conditions in Proposition \ref{QBcor} must have an indefinite principal submatrix 
	$$\begin{pmatrix} 1 & -1 & -1 \\ -1 & 1 & -1 \\ -1 & -1 & 1 \end{pmatrix},$$ and this only holds when $r\geq 3$.
\end{rem}

\section{The \emph{sos} Cases in Theorem~\ref{sos thm}}\label{pr-sos}

In this section, we establish the \emph{sos} cases in Theorem~\ref{sos thm}. 
By the SDP characterization obtained earlier, it suffices to construct, for each admissible multiplicity pair, a feasible matrix $B$ satisfying the constraints in Theorem~\ref{thm:SDP-feasible-B}. 
In other words, we construct explicit matrices $B$ for the three cases listed in Corollary~\ref{reduced cases}, thereby proving that $G_F$ is \emph{sos} in these situations.

Technically, we first derive a set of necessary conditions that any matrix $B$ satisfying the constraints of Theorem~\ref{thm:SDP-feasible-B} must fulfill. 
Guided by these conditions, we then construct specific candidate matrices $B$ and verify that they indeed satisfy all the required constraints. 
The three multiplicity cases are treated separately in the following subsections.

\subsection{Constructing Feasible Matrices for $(m,l)=(1,k+2)$}\label{sos (1,k+2)}
The case $m=1$ is degenerate. Let $E_0 := I_l$, and let the Clifford system $\{P_0,P_1\}$ on $\mathbb{R}^{2l}$ be defined as in \eqref{Cliffordsys-alg}. 

In the present case, for any $1\leq q\leq l$, $R_q$ is a row vector which, according to definition \eqref{Define R_q}, is the $q$-th row of $E_0$. Hence, $R_q = v_q E_0 = v_q$. To facilitate referencing in later parts of the paper, we introduce the notation $R_q(1,l)$ for $R_q$ in the present context, i.e., 
\begin{equation}\label{Define R_q(1,l)}
	R_q(1,l):=R_q=v_q,\quad R(1,l):=R=(v_1,\cdots,v_l),\quad 1\leq q\leq l.
\end{equation}

Suppose $G_F(x) = |x|^4 - \sum_{\alpha=0}^{1} \langle P_\alpha x, x \rangle^2$ is \emph{sos}. By Proposition~\ref{QBcor} there exists a positive semidefinite matrix $B$ fulfilling \eqref{b_ijij}--\eqref{b_ji} such that $R_i B_{ij}=R_j$ for all $i,j$. For $i \ne j$, Lemma~\ref{properties of B} tells us that $B_{ij}$ is skew‑symmetric. The relation $R_i B_{ij} = R_j$, i.e., $v_i B_{ij} = v_j$, therefore forces the $(i,j)$--entry of $B_{ij}$ to be $1$.

We now construct the simplest possible matrix $B$ satisfying these conditions. 
Let 
\begin{equation}\label{B(1,l)}
	B(1,l) := (B_{ij})_{i,j=1}^l
\end{equation} 
be the block matrix defined by
\[
B_{ii} = I_l \quad \text{and} \quad B_{ij} = E_{ij} - E_{ji} \qquad (1 \le i \ne j \le l),
\]
where $E_{ij}$ denotes the $l \times l$ matrix unit. 
\begin{prop}\label{prop:B(1,l)}
	 
	The matrix $B(1,l)$ satisfies all the conditions of Proposition~\ref{QBcor}. Moreover, 
	\[
	\mathrm{rank}\bigl(B(1,l)\bigr) = \frac{l(l-1)}{2} + 1.
	\]
\end{prop}

\begin{proof}
	It is straightforward to verify by direct computation that $B(1,l)$ satisfies
	conditions~\eqref{Bentry} and~\eqref{BRrelation} in Proposition~\ref{QBcor}.  
	It remains to establish condition~\eqref{Bpsd} and to compute the rank of $B(1,l)$.
	
	Observe that
	\[
	B(1,l)
	= \sum_{i,j} E_{ij} \otimes E_{ij}
	+ \sum_{i} \sum_{j>i}
	\bigl(
	E_{ii} \otimes E_{jj}
	- E_{ij} \otimes E_{ji}
	- E_{ji} \otimes E_{ij}
	+ E_{jj} \otimes E_{ii}
	\bigr).
	\]
	We write
	\[
	B(1,l) = \widetilde{B} + \sum_{i} \sum_{j>i} \widehat{B}_{ij}.
	\]
	
	The matrix $\widetilde{B}$ has exactly $l^2$ nonzero entries, forming an
	$l \times l$ all-ones submatrix, and is therefore positive semidefinite with
	rank one.  
	On the other hand, for each $i<j$, the matrix $\widehat{B}_{ij}$ contains only
	four nonzero entries, forming a $2 \times 2$ principal submatrix
	\[
	\begin{pmatrix}
		1 & -1 \\
		-1 & 1
	\end{pmatrix},
	\]
	which is positive semidefinite and of rank one.
	
	Since the supports of $\widetilde{B}$ and the matrices $\widehat{B}_{ij}$
	are mutually orthogonal, it follows that $B(1,l)$ is positive semidefinite and
	\[
	\mathrm{rank}\bigl(B(1,l)\bigr)
	= 1 + \binom{l}{2}
	= \frac{l(l-1)}{2} + 1.
	\]
	This completes the proof.
\end{proof}

In summary, the matrix $B(1,l)$ constructed above fulfills all three conditions of Proposition~\ref{QBcor}. Therefore, we conclude that $G_F$ is \emph{sos} for $(m,l) = (1, k+2)$ ($\forall \ k \in \mathbb{N}^+$).

\subsection{Constructing Feasible Matrices for $(m,l)=(2,2k+2)$}\label{sos (2,2k+2)}
Recall that $\iota (\oi) = \begin{pmatrix} 0 & -1 \\ 1 & 0 \end{pmatrix}$ by \eqref{iota}.
Let $E_0:=I_{l}$. Clifford algebra $C_{1}$ has a complex matrix representation on $\mathbb{C}^{k+1}$ given by $\oi I_{k+1}$. Let $E_1$ denote the corresponding real matrix of $-\oi I_{k+1}$, i.e., 
 $$E_1:=-\iota_{k+1}(\oi I_{k+1})=-I_{k+1}\otimes \iota (\oi),$$
 where $\iota_{k+1}$ is defined as in \eqref{iota_k} (here a negative sign is added for computational convenience). Consider the Clifford system $\{P_0,P_1,P_2\}$ on $\mathbb{R}^{2l}$ obtained by substituting $E_1$ into  \eqref{Cliffordsys-alg}.

Unless otherwise stated, we adopt the following index ranges in this subsection:
$$1\leq i,j\leq l,\quad 1\leq s,h\leq k+1,\quad t=0,1. $$
Let $\{E_{sh}\}$ denote the $(k+1) \times (k+1)$ standard matrix basis of $M(k+1,\mathbb{R})$. Let \[
L_1:=I_l,\qquad L_s:=(E_{1s}-E_{s1})\otimes I_2 \ \ (2\le s\le k+1).
\]

In the present case, for any $1\leq q\leq l$, $R_q$ is a $2 \times l$ matrix whose $\alpha$-th row (for $\alpha =1, 2$) is given by the $q$-th row of $E_{\alpha-1}$, according to definition \eqref{Define R_q}. 
Therefore,
\[
R_1=\begin{pmatrix} I_2 & O_2 & \cdots & O_2 \end{pmatrix},\quad
R_2=-\begin{pmatrix} \iota(\oi) & O_2 & \cdots & O_2 \end{pmatrix},\quad
R_{2s-t}=R_{2-t}L_s,\ \ \forall\, s,t,
\]
where $O_2$ denotes the $2\times2$ zero matrix.
To facilitate referencing in later parts of the paper, we introduce the notation $R_q(2,l)$ for $R_q$ in the present context, i.e., 
\begin{equation}\label{Define R_q(2,l)}
	R_q(2,l):=R_q,\quad R(2,l):=R=(R_1,\cdots,R_l),\quad  1\leq q\leq l, 
\end{equation}
where $R_q$ is as above.

Suppose $G_F(x) = |x|^4 - \sum_{\alpha=0}^{2} \langle P_\alpha x, x \rangle^2$ is \emph{sos}. By Proposition~\ref{QBcor} there exists a positive semidefinite matrix $B$ fulfilling \eqref{b_ijij}--\eqref{b_ji} such that $R_i B_{ij}=R_j$ for all $i,j$. The condition $R_i B_{ij} = R_j$ for all $i,j$ is equivalent to: 
\begin{align}
	&R_{2s-1} B_{2s-1,j} = R_j, \label{R2s-1} \\
	&R_{2s} B_{2s,j} = R_j, \label{R2s}
\end{align}
for any $s,j$.
Furthermore, using the relation $\iota(\oi)R_{2s} = R_{2s-1}$ and left-multiplying \eqref{R2s} by $\iota(\oi)$, we obtain the equivalent form 
\begin{equation}\label{equi form of R2s}
	R_{2s-1} B_{2s,j} = \iota(\oi)R_j,
\end{equation}
for any $s,j$. Hence, $R_i B_{ij} = R_j$ for all $i,j$ is equivalent to \eqref{R2s-1} and \eqref{equi form of R2s}. This means that the matrix formed by the $(2s-1)$-th and $2s$-th rows of $B_{2s-1,j}$ is exactly $R_j$, and the matrix formed by the $(2s-1)$-th and $2s$-th rows of $B_{2s,j}$ is
$$
\iota(\oi)R_j=
\begin{cases}
	-R_{2h},\ j=2h-1,\\
	R_{2h-1},\ j=2h.
\end{cases}
$$
On the other hand, taking the second line of \eqref{R2s-1} and applying Lemma \ref{Btauthm}, it is easy to deduce that $B_{2s-1,2s}=-\tau_{2s}(E_1)$.

Combining the above conditions, we choose a symmetric matrix 
\begin{equation}\label{B(2,l)}
	B(2,l) := (B_{ij})_{i,j=1}^l,
\end{equation}
such that $B_{ii}=I_l$, $B_{2s-1,2s}=-\tau_{2s}(E_1)$, and 
$$
\begin{pmatrix}
	B_{2s-1,2h-1} & B_{2s-1,2h} \\
	B_{2s,2h-1} & B_{2s,2h}
\end{pmatrix}
=
\begin{pmatrix}
	(E_{sh}-E_{hs})\otimes I_2 & -(E_{sh}+E_{hs})\otimes \iota(\oi) \\
	(E_{sh}+E_{hs})\otimes \iota(\oi) & (E_{sh}-E_{hs})\otimes I_2
\end{pmatrix}.
$$
for all $i,s,h$ with $s\neq h$.

\begin{prop}\label{prop:B(2,l)}
	The matrix $B(2,l)$ satisfies all the conditions of Proposition~\ref{QBcor}. 
	Moreover,
	\[
	\mathrm{rank}\bigl(B(2,l)\bigr)
	= \frac{l(l-2)}{4} + 2.
	\]
\end{prop}

\begin{proof}
	We first verify conditions (\ref{Bentry}) and (\ref{BRrelation}).
	A direct computation shows that both \eqref{R2s-1} and \eqref{R2s} hold,
	and hence condition (\ref{BRrelation}) is satisfied.
	All parts of condition (\ref{Bentry}) follow from routine calculations,
	except for \eqref{b_ji}.
	To establish \eqref{b_ji}, it suffices to verify its equivalent form
	\[
	v_j B_{ik} = - v_i B_{jk}, \qquad j \ne i,
	\]
	which can be checked by a case-by-case discussion of the indices $i,j,k$.
	
	We now turn to condition (\ref{Bpsd}). By Lemma~\ref{BBthm} and the skew-symmetry relation
	$B_{2s,2s-1}=-B_{2s-1,2s}$, for every $1\le j\le l$ one has
	\[
	B_{2s,j}
	= B_{2s,2s-1} B_{2s-1,j}
	= -\, B_{2s-1,2s} B_{2s-1,j}.
	\]
	We perform a congruence transformation on $B(2,l)$ at the level of block rows and block columns.
	For each $s$, we left-multiply the $(2s-1)$-st block row of $B(2,l)$ by the block
	$B_{2s-1,2s}$ and add it to the $2s$-th block row, and simultaneously
	right-multiply the $(2s-1)$-st block column by $B_{2s-1,2s}^{T}$ and add it
	to the $2s$-th block column.
	Under this transformation, all even-numbered block rows and block columns become zero.
	
	Let $\widetilde{B}$ denote the submatrix formed by the odd-numbered block rows and
	block columns of the resulting matrix.
	Then $\widetilde{B}$ admits the block representation
	\[
	\widetilde{B} = \widehat{B} \otimes I_2,
	\]
	where $\widehat{B} = (\widehat{B}_{sh})_{s,h=1}^{k+1}$ satisfies
	$\widehat{B}_{ss} = I_{k+1}$ and $\widehat{B}_{sh} = E_{sh} - E_{hs}$ for all $1\leq s\neq h\leq k+1$.
	In particular, $\widehat{B}$ coincides with the matrix $B(1,k+1)$
	introduced in \eqref{B(1,l)}.
	Hence $\widehat{B}$ is positive semidefinite.
	It follows that there exists a matrix $G$ such that $\widehat{B} = G^{T} G$, and therefore
	\[
	\widetilde{B}
	= (G^{T} G) \otimes I_2
	= (G \otimes I_2)^{T} (G \otimes I_2),
	\]
	which shows that $\widetilde{B}$ is positive semidefinite.
	Since congruence transformations preserve positive semidefiniteness,
	we conclude that $B(2,l)$ is positive semidefinite.
	
	Finally, since
	\[
	\mathrm{rank}(\widetilde{B})
	= \mathrm{rank}(\widehat{B}) \cdot \mathrm{rank}(I_2),
	\]
	and Proposition~\ref{prop:B(1,l)} yields
	$\mathrm{rank}(\widehat{B}) = 1 + \binom{k+1}{2}$,
	we obtain
	\[
	\mathrm{rank}\bigl(B(2,l)\bigr)
	= 2\Bigl(1 + \binom{k+1}{2}\Bigr)
	= \frac{l(l-2)}{4} + 2.
	\]
	This completes the proof.
\end{proof}

The matrix $B(2,l)$ constructed above meets every requirement of Proposition~\ref{QBcor}. Consequently, $G_F$ must be \emph{sos} for  $(m,l) = (2, 2k+2)$, where $k$ is any positive integer.

\subsection{The Unique Feasible Matrix for $(m,l)=(6,8)$}\label{m=6}
The Dirac matrices are defined in terms of the Pauli matrices (see \eqref{Pauli}) as follows:
$$
\gamma_0 := \sigma_3 \otimes I_2, \quad 
\gamma_j := \oi \sigma_2 \otimes \sigma_j ~ (j=1,2,3), \quad 
\gamma_5 := \oi \gamma_0\gamma_1\gamma_2\gamma_3.
$$

A $4 \times 4$ complex matrix representation of Clifford algebra $C_{5}$ (\textit{cf.} \cite{Wilson2021}) is 
$$\oi \gamma_0,\quad \gamma_1,\quad \gamma_2,\quad \gamma_3,\quad \oi \gamma_5.$$
Let $E_1,\cdots, E_5$ denote their corresponding real matrices, i.e., 
\begin{equation*}\label{cl of m=6}
	E_1:=\iota_4(\oi\gamma_0),\  E_2:=\iota_4(\gamma_1),\  E_3:=\iota_4(\gamma_2),\  E_4:=\iota_4(\gamma_3),\  E_5:=\iota_4(\oi\gamma_5),
\end{equation*}
where $\iota_{4}$ is defined as in \eqref{iota_k}. Consider the Clifford system $\{P_0,\cdots,P_6\}$ on $\mathbb{R}^{16}$ obtained by substituting $E_1,\cdots, E_5$ into  \eqref{Cliffordsys-alg}.

Let $E_0:=I_8$ and $T:=I_4\otimes \oi \sigma_2$. In the present case, for any $1\leq q\leq 8$, $R_q$ is a $6 \times 8$ matrix whose $\alpha$-th row (for $1\leq \alpha \leq 6$) is given by the $q$-th row of $E_{\alpha-1}$, according to definition \eqref{Define R_q}. Therefore
$$
R_1=
\begin{pmatrix}
	\sigma_3 & O_2 & O_2 & O_2 \\
	O_2 & O_2 & O_2 & I_2 \\
	O_2 & O_2 & \sigma_3 & O_2
\end{pmatrix},\quad
R_2=R_1 T,\quad
R_3=
\begin{pmatrix}
	O_2 & \sigma_3 & O_2 & O_2 \\
	O_2 & O_2 & \sigma_3 & O_2 \\
	O_2 & O_2 & O_2 & -I_2
\end{pmatrix},\quad
R_4=R_3 T,
$$
$$
R_5=
\begin{pmatrix}
	O_2 & O_2 & I_2 & O_2 \\
	O_2 & -I_2 & O_2 & O_2 \\
	-I_2 & O_2 & O_2 & O_2
\end{pmatrix},\quad
R_6=R_5 T,\quad
R_7=
\begin{pmatrix}
	O_2 & O_2 & O_2 & I_2 \\
	-\sigma_3 & O_2 & O_2 & O_2 \\
	O_2 & \sigma_3 & O_2 & O_2
\end{pmatrix},\quad
R_8=R_7 T.
$$
where $O_2$ denotes the $2\times 2$ zero matrix. To facilitate referencing in later parts of the paper, we introduce the notation $R_q^{(6)}$ for $R_q$ in the present context, i.e., 
\begin{equation}\label{Define R_q^6}
	R_q^{(6)}:=R_q,\quad R^{(6)}:=R=(R_1,\cdots,R_8),\quad 1\leq q\leq 8,
\end{equation}
where $R_q$ is as above.

Suppose $G_F(x) = |x|^4 - \sum_{\alpha=0}^{6} \langle P_\alpha x, x \rangle^2$ is \emph{sos}. By Proposition~\ref{QBcor} there exists a positive semidefinite matrix $B$ fulfilling \eqref{b_ijij}--\eqref{b_ji} such that $R_i B_{ij}=R_j$ for all $i,j$. From the second row of $R_1 B_{1k}=R_k$ for all $k$, we obtain $v_2 B_{1k}=-w_2 R_k$ for all $k$. According to Lemma \ref{Btauthm}, it follows that $B_{12}=\tau_1 (E_1)=\tau_2 (E_1)$. Similarly, from rows $3,4,5,6$ of $R_1 B_{1k}=R_k$ for all $k$, we deduce
$$
B_{17}=-\tau_1 (E_2),\quad  B_{18}=-\tau_1 (E_3),\quad  B_{15}=-\tau_1 (E_4),\quad  B_{16}=\tau_1 (E_5).
$$
Moreover, from rows $3,4$ of $R_5 B_{5k}=R_k$ for all $k$, we get $$B_{53}=\tau_5 (E_2), \quad B_{54}=\tau_5 (E_3).$$
Note that, specifically for any matrix $B_{ij}$ obtained from Lemma \ref{Btauthm}, there are two equivalent representations: $B_{ij}=\mp \tau_{i}(E_{\alpha-1})=\mp \tau_{j}(E_{\alpha-1})$. In the subsequent calculations, we may alternate between the two forms for convenience.

Let $E_6:=E_2E_4,\ E_7:=E_3E_4$. Since $B_{15}$ is orthogonal, by Lemma \ref{BBthm} we obtain
\begin{align*}
	&B_{13}=B_{15}B_{53}=-\tau_1 (E_4)\tau_5 (E_2)=-\tau_5 (E_4)\tau_5 (E_2)=I_{8}^{(5)} E_2 E_4 I_{8}^{(5)}=\tau_5 (E_6),\\
	&B_{14}=B_{15}B_{54}=-\tau_1 (E_4)\tau_5 (E_3)=-\tau_5 (E_4)\tau_5 (E_3)=I_{8}^{(5)} E_3 E_4 I_{8}^{(5)}=\tau_5 (E_7).
\end{align*}
$B= (B_{ij})_{i,j=1}^8$ is an $8\times8$ symmetric block matrix. We have now determined its first block row, denoted by $V^{(6)}$:
\begin{align}
	V^{(6)}:&=(B_{1j})_{j=1}^8\notag\\
	&=
	\begin{pmatrix}
		I_8 & \tau_{1}(E_{1}) & \tau_{5}(E_{6}) & \tau_{5}(E_{7}) & -\tau_{1}(E_{4}) & \tau_{1}(E_{5}) & -\tau_{1}(E_{2}) & -\tau_{1}(E_{3})
	\end{pmatrix}.\label{eq:V6-block-row}
\end{align}
Each block of $V^{(6)}$ is an orthogonal matrix, and all blocks are skew-symmetric except for the first block. Therefore, applying Lemma \ref{BBthm}, it follows that $B_{ij} = B_{i1}B_{1j} = B_{1i}^T B_{1j}$ for all $1 \leq i,j \leq 8$. This shows that $B$ is completely determined by its first block row.

Let
\begin{equation}\label{Define B^6}
	B^{(6)} := (V^{(6)})^{T} V^{(6)} .
\end{equation}
Then $B = (V^{(6)})^{T} V^{(6)} = B^{(6)}$.
Since $V^{(6)}$ has $l=8$ rows and full row rank, it follows that
\[
\mathrm{rank}\bigl(B^{(6)}\bigr)
= \mathrm{rank}\bigl(V^{(6)}\bigr)
= l = 8 .
\]

Finding a matrix $B$ that satisfies the three conditions of Proposition \ref{QBcor} is, in essence, a semidefinite programming problem. The discussion above demonstrates that any feasible solution of the SDP must be $B^{(6)}$. It remains, of course, to verify that $B^{(6)}$ does indeed fulfill all three conditions stipulated in Proposition \ref{QBcor}.

By \eqref{eq:V6-block-row} and \eqref{Define B^6}, if
\[
B^{(6)}=\bigl(B^{(6)}_{ij}\bigr)_{i,j=1}^{8},
\]
then
\[
B^{(6)}_{ij}=B_{1i}^TB_{1j},\qquad 1\leq i,j\leq 8.
\]
In particular, $B^{(6)}\succeq 0$, so condition \eqref{Bpsd} of Proposition~\ref{QBcor} is satisfied.

Using the explicit block row \eqref{eq:V6-block-row}, together with the Clifford-algebra relations for $E_1,\ldots,E_5$, the definitions $E_6=E_2E_4$, $E_7=E_3E_4$, and the definition of $\tau_k$ in \eqref{tau_{k}}, a blockwise computation gives
\[
B_{1i}^TB_{1i}=I_8,\qquad
B_{1i}^TB_{1j}+B_{1j}^TB_{1i}=0\quad (i\neq j).
\]
Equivalently,
\[
B^{(6)}_{ii}=I_8,\qquad
B^{(6)}_{ij}=-(B^{(6)}_{ij})^T\quad (i\neq j).
\]
Hence the block relations in \eqref{Bentry} are satisfied.

It remains to check \eqref{BRrelation}. From the explicit formulas for $R_1,\ldots,R_8$ in \eqref{Define R_q^6} and the block row \eqref{eq:V6-block-row}, we have
\[
R_j^{(6)}=R_1^{(6)}B_{1j},\qquad 1\leq j\leq 8.
\]
Therefore, for all $1\leq i,j\leq 8$,
\[
R_i^{(6)}B^{(6)}_{ij}
=
R_1^{(6)}B_{1i}B_{1i}^TB_{1j}
=
R_1^{(6)}B_{1j}
=
R_j^{(6)}.
\]
Thus \eqref{BRrelation} holds. Consequently $B^{(6)}$ satisfies all the conditions in Proposition~\ref{QBcor}.

By Proposition \ref{QBcor}, the corresponding form $G_F$ is \emph{sos} for the representative Clifford system with $(m,l)=(6,8)$.

At this stage, cases (1) and (2) of Corollary \ref{reduced cases} have been verified. Hence, the proof of Theorem \ref{sos thm} is complete.

\begin{rem}
    For the block matrix $B^{(6)} = (B_{ij})_{i,j=1}^8$, using the properties of $B^{(6)}$ we obtain
    \[
    B_{1i}B_{1k} + B_{1k}B_{1i} = -(B_{ik} + B_{ki}) = -2\delta_{ik} I_8 \quad \text{for all } i,k \geq 2,
    \]
    where $\delta_{ik}$ is the Kronecker delta. This shows that $\{B_{12}, \cdots, B_{18}\}$ generate a Clifford algebra $C_{7}$ on $\mathbb{R}^8$.
\end{rem}

\section{Ranks of \emph{sos} Representations and the Proof of Theorem~\ref{rank thm}}\label{Rank proof}

In this section, we prove Theorem~\ref{rank thm} by combining a general discussion of \emph{sos} representation ranks with the SDP characterization established earlier for $G_F$. We first develop, in Subsection~\ref{Ranks via SDP}, a general framework relating the ranks of \emph{sos} representations of a polynomial to the ranks of positive semidefinite Gram matrices, or equivalently, to the ranks of feasible matrices of the associated semidefinite program. We then apply this framework to the OT-FKM type forms $G_F$, and determine the possible ranks in each \emph{sos} case, thereby completing the proof of Theorem~\ref{rank thm}.

\subsection{Ranks of \emph{sos} Representations via SDP}\label{Ranks via SDP}

For a nonnegative polynomial $p(x)$ of degree $2d$ with an \emph{sos} representation
\[
p(x)=\sum_{k=1}^{N}p_k(x)^2,
\]
the number of linearly independent polynomials among $\{p_1,\dots,p_N\}$ is called the \emph{rank of the sos representation}, denoted by $r$. Clearly, $1\le r\le N$, and this value depends on the chosen representation.

Given a column vector of polynomials $z(x)=(z_1(x),\dots,z_q(x))^T$ whose components are linearly independent, a symmetric matrix $S$ satisfying
\[
p(x)=z(x)^T S z(x)
\]
is called a \emph{Gram matrix} of $p(x)$ with respect to $z(x)$. In particular, let
\begin{equation}\label{z(x)}
	z(x):=\bigl(x^\alpha\bigr)_{|\alpha|\le d}
\end{equation}
be the vector of all monomials of degree at most $d$. By Proposition~\ref{sostoSDP}, the polynomial $p(x)$ is \emph{sos} if and only if there exists a positive semidefinite Gram matrix of $p(x)$ with respect to $z(x)$.

Indeed, given an \emph{sos} representation, write each $p_k(x)=V_k^T z(x)$ and set $V=(V_1,\dots,V_N)^T$. Then
\[
p(x)=\sum_{k=1}^{N}(V_k^T z(x))^2=z(x)^T(V^T V)z(x),
\]
so $S=V^T V$ is a positive semidefinite Gram matrix. Moreover, the rank of the representation equals the rank of $S$:
\begin{equation}\label{r=rank(S)}
	\mathrm{rank}(S)=\mathrm{rank}(V)=r.
\end{equation}

Related but distinct from the rank of a specific \emph{sos} representation is the \emph{sos} rank, a notion that has been more extensively studied in the general theory of \emph{sos} decompositions. The \emph{sos rank} of $p(x)$, denoted by $\mathrm{rank}(p)$, is defined as
\[
\mathrm{rank}(p):=\min\Bigl\{N:\; p(x)=\sum_{k=1}^{N}p_k(x)^2\Bigr\},
\]
namely, the minimum number of squares in any \emph{sos} representation of $p(x)$.

The next proposition identifies the minimum possible representation rank with the \emph{sos} rank.

\begin{prop}
	Let
	\[
	r_{\min}(p):=\min\{r:\text{$p(x)$ admits an \emph{sos} representation of rank $r$}\}.
	\]
	Then
	\[
	r_{\min}(p)=\mathrm{rank}(p).
	\]
\end{prop}

\begin{proof}
	Let
	\[
	p(x)=\sum_{k=1}^{N}p_k(x)^2
	\]
	be an \emph{sos} representation with the minimum number of squares, so that $N=\mathrm{rank}(p)$. The rank of this representation is at most $N$, hence 
	$$r_{\min}(p)\le N=\mathrm{rank}(p).$$
	
	Conversely, let
	\[
	p(x)=\sum_{k=1}^{N}p_k(x)^2
	\]
	be any \emph{sos} representation of rank $r$. Choose linearly independent polynomials $q_1(x),\dots,$ $q_r(x)$ spanning $\mathrm{span}\{p_1(x),\dots,p_N(x)\}$. Then $$p_k(x)=\sum_{i=1}^{r} c_{ki}q_i(x)$$ for some matrix $C=(c_{ki})\in\mathbb{R}^{N\times r}$. Writing $q(x)=(q_1(x),\dots,q_r(x))^T$, we obtain
	\[
	p(x)=\sum_{k=1}^{N}p_k(x)^2=q(x)^T C^T C\, q(x).
	\]
	Since the representation has rank $r$, the matrix $C$ has rank $r$. Therefore $C^T C$ is positive definite, so there exists an invertible matrix $M\in\mathbb{R}^{r\times r}$ such that $C^T C=M^T M$. Let $\tilde q(x):=Mq(x)$, with components $\tilde q_1(x),\dots,\tilde q_r(x)$. Then
	\[
	p(x)=q(x)^T M^T M q(x)=\tilde q(x)^T \tilde q(x)=\sum_{i=1}^{r}\tilde q_i(x)^2.
	\]
	Hence $p(x)$ admits an \emph{sos} representation with exactly $r$ squares, and so $\mathrm{rank}(p)\le r$. Since this holds for every \emph{sos} representation of rank $r$, we obtain $$\mathrm{rank}(p)\le r_{\min}(p).$$
	
	Combining the two inequalities yields $r_{\min}(p)=\mathrm{rank}(p)$.
\end{proof}

We now turn to a broader question: what are all possible ranks that can occur among \emph{sos} representations of $p(x)$? The following theorem answers this question by linking \emph{sos} representation ranks to the ranks of positive semidefinite Gram matrices of $p(x)$ with respect to $z(x)$. Equivalently, it identifies $\mathcal{R}(p)$ with the set of ranks attained by feasible solutions of the semidefinite program in Proposition~\ref{sostoSDP}.

\begin{thm}\label{rank equivalence}
	For a polynomial $p(x)$ of degree $2d$, let $\mathcal{R}(p)$ denote the set of all possible ranks of its \emph{sos} representations. Then
	\[
	\mathcal{R}(p) = \{ \mathrm{rank}(S) : S \succeq 0, \; p(x) = z(x)^T S z(x) \},
	\]
	where $ z(x) $ is defined as in \eqref{z(x)}.
\end{thm}

\begin{proof}
	First let $r\in\mathcal{R}(p)$. Then $p(x)$ admits an \emph{sos} representation of rank $r$. As above, this representation produces a positive semidefinite Gram matrix $S$ satisfying
	\[
	p(x)=z(x)^T S z(x),
	\]
	and \eqref{r=rank(S)} gives $\mathrm{rank}(S)=r$. Hence
	\[
	\mathcal{R}(p)\subseteq \{ \mathrm{rank}(S) : S \succeq 0,\; p(x)=z(x)^T S z(x)\}.
	\]
	
	Conversely, let $S\succeq 0$ satisfy
	\[
	p(x)=z(x)^T S z(x),
	\]
	and let $\mathrm{rank}(S)=r$. Since $S\succeq 0$ and $\mathrm{rank}(S)=r$, there exists a matrix $V\in\mathbb{R}^{r\times q}$ of full row rank $r$ such that
	$
	S=V^T V
	$.
	 Let $V_1,\dots,V_r$ be the rows of $V$, and define
	\[
	q_i(x):=V_i^T z(x),\qquad i=1,\dots,r.
	\]
	Then
	\[
	p(x)=z(x)^T V^T V z(x)=\sum_{i=1}^{r} q_i(x)^2.
	\]
	Since the rows of $V$ are linearly independent and the components of $z(x)$ are linearly independent, the polynomials $q_1(x),\dots,q_r(x)$ are linearly independent. Thus this is an \emph{sos} representation of rank $r$, and therefore
	\[
	\{ \mathrm{rank}(S) : S \succeq 0,\; p(x)=z(x)^T S z(x)\}\subseteq \mathcal{R}(p).
	\]
	The two inclusions imply the desired equality.
\end{proof}

Theorem~\ref{rank equivalence} has an immediate consequence:

\begin{cor}
	If the positive semidefinite Gram matrix of $p(x)$ with respect to $z(x)$ is unique, then all \emph{sos} representations of $p(x)$ have the same rank. Equivalently, $\mathcal{R}(p)$ consists of a single rank.
\end{cor}

\subsection{Rank Sets of \emph{sos} Representations of $G_F$}\label{sec:sos-rank-GF}

We now focus on the ranks of \emph{sos} representations for the specific nonnegative polynomial 
\[
G_F(x)=|x|^4-\sum_{\alpha=0}^m\langle P_\alpha x,x\rangle^2,
\]
constructed from an OT-FKM type isoparametric polynomial. Here $\{P_0,\dots,P_m\}$ is a Clifford system on $\mathbb{R}^{2l}$ whose algebraic representation is given by \eqref{Cliffordsys-alg}, with the associated skew‑symmetric matrices $E_1,\dots,E_{m-1}$ generating a Clifford algebra on $\mathbb{R}^l$. By Theorem~\ref{sos thm}, $G_F$ is a sum of squares precisely when the multiplicity pair $(m_+, m_-)=(m, l-m-1)$ belongs to the list 
\[
(1,k),\;(2,2k-1),\;(3,4),\;(4,3)^I,\;(5,2),\;(6,1),\qquad k\in\mathbb{N}^+,
\]
where the superscript $I$ denotes the indefinite class. In the following we always assume that $(m,l)$ is one of these admissible pairs, so that $G_F$ admits at least one \emph{sos} representation.

By Lemma~\ref{SDP}, $G_F$ is \emph{sos} if and only if there exists a positive semidefinite matrix $Q$ satisfying $G_F(x)=X^{T}QX$. Recall the matrices $R_1,\cdots,R_l \in M(m\times l,~\R)$ defined in \eqref{Define R_q} and the aggregated matrix $R := (R_1, \dots, R_l) \in M(m \times l^2,~\R)$ from \eqref{Define R}. 
For the matrix $R$, define
\begin{equation}\label{B(R)}
	\mathcal{B}(R):=\{ B\succeq0 \; |\; 
	R_iB_{ij}=R_j\;(1\le i,j\le l),\;\eqref{b_ijij}\text{--}\eqref{b_ji}\text{ hold}\}.
\end{equation}
According to Proposition~\ref{QBcor}, the existence of $Q$ is equivalent to the existence of an $l^2 \times l^2$ matrix $B\in \mathcal{B}(R)$.

For an \emph{sos} representation $G_F = \sum_{k=1}^{N} p_k(x)^2$, let $r$ denote its rank, i.e., the number of linearly independent polynomials among $\{p_1,\dots,p_N\}$.  By \eqref{r=rank(S)}, $r$ equals the rank of a corresponding positive semidefinite Gram matrix $S$ with respect to $z(x)$.  For the quartic form $G_F$ the natural choice of the monomial basis is the vector $X$ of all quadratic monomials (see Remark~\ref{rem of sostoSDP}); consequently the Gram matrix becomes exactly the matrix $Q$ appearing in Lemma~\ref{SDP}.  Hence  
$
r = \mathrm{rank}(Q).
$
Moreover, Proposition~\ref{R(Q)=R(B-RTR)} gives $$r=\mathrm{rank}(Q) = \mathrm{rank}(B - R^{T}R).$$  Therefore, combining Proposition~\ref{QBcor} and Theorem~\ref{rank equivalence} we obtain
\begin{align}
	\mathcal{R}(G_F)&=\{ \mathrm{rank}(Q) : Q \succeq 0, \; G_F(x) = X^T Q X \}\notag\\
    &=\{ \mathrm{rank}(B - R^{T}R)  :  B\in\mathcal{B}(R)\}.\label{rank(B - RTR)}
\end{align}

The main result concerning $r \in \mathcal{R}(G_F)$ is summarized in Theorem~\ref{rank thm}.  
To prove this theorem, we first present several lemmas.  
Initially, we establish the invariance of $\mathcal{R}(G_F)$ under geometric equivalence of Clifford systems. 

\begin{lem}\label{R(G_F) invariant}
	Let $\{P_0,\dots,P_m\}$ and $\{P'_0,\dots,P'_m\}$ be two geometrically equivalent Clifford systems on $\mathbb{R}^{2l}$, and denote
	\[
	G_F(x):=|x|^4-\sum_{\alpha=0}^{m}\langle P_\alpha x,x\rangle^2,\qquad 
	G'_F(x):=|x|^4-\sum_{\alpha=0}^{m}\langle P'_\alpha x,x\rangle^2 .
	\]
	Then the sets of possible ranks of their \emph{sos} representations coincide, i.e.\ 
	\[
	\mathcal{R}(G_F)=\mathcal{R}(G'_F).
	\]
\end{lem}
\begin{proof}
	As shown in the proof of Lemma~\ref{geom. equi. with sos}, there exists an orthogonal matrix $W\in O(\mathbb{R}^{2l})$ such that
	\[
	G'_F(x)=G_F(Wx)\qquad\text{for all }x\in\mathbb{R}^{2l}.
	\]
	Assume $r\in\mathcal{R}(G_F)$. Then there exists an \emph{sos} representation $G_F(x)=\sum_{k=1}^{N}p_k(x)^2$ with rank $r$.  
	Substituting $x\mapsto Wx$ gives
	\[
	G'_F(x)=G_F(Wx)=\sum_{k=1}^{N}p_k(Wx)^2,
	\]
	which is an \emph{sos} representation of $G'_F$ whose rank is again $r$ because the polynomials $\{p_k(Wx)\}$ are linearly independent iff $\{p_k(x)\}$ are. Hence $r\in\mathcal{R}(G'_F)$.  
	The converse inclusion $\mathcal{R}(G'_F)\subset\mathcal{R}(G_F)$ follows by the same argument applied to the inverse transformation $W^{-1}$.  
	Therefore $\mathcal{R}(G_F)=\mathcal{R}(G'_F)$.
\end{proof}

Consequently, when describing $\mathcal{R}(G_F)$ for a given admissible pair $(m_+, m_-)$, it suffices to consider a single representative from each geometric equivalence class.

We now turn to a special case. Consider a Clifford system $\{P_0,\cdots,P_m\}$ expressed as in \eqref{Cliffordsys-alg}, with associated skew‑symmetric matrices $\{E_1,\cdots,E_{m-1}\}$. For any integer $m' < m$, the smaller Clifford system $\{P_0,$ $\cdots,$ $P_{m'}\}$ is obtained by taking the first $m'$ matrices; consequently, its corresponding skew‑symmetric matrices are simply $\{E_1,\cdots,E_{m'-1}\}$. Let $E_0 = I_l$. Using the notation in \eqref{Define R_q}, let $R_j$ (resp. $R'_j$) be the $m \times l$ (resp. $m' \times l$) matrix whose $\alpha$-th row is $v_j E_{\alpha-1}$ for $\alpha = 1,\cdots,m$ (resp. $\alpha = 1,\cdots,m'$). Set $R = (R_1,\cdots,R_l)$ and $R' = (R'_1,\cdots,R'_l)$.  Here $R'$ is simply formed by taking the first $m'$ rows of $R$.  The following lemma concerns the relation between $\mathcal{B}(R)$ and $\mathcal{B}(R')$.

\begin{lem}\label{B(R)inB(R')}
For any integer $m' < m$, let $R = (R_1,\cdots,R_l)\in M(m \times l^2,~\R)$ and let $R' = (R'_1,\cdots,R'_l)\in M(m' \times l^2,~\R)$ be the submatrix consisting of the first $m'$ rows of $R$. Then $$\mathcal{B}(R) \subseteq \mathcal{B}(R'),$$ where $\mathcal{B}(\cdot)$ is defined in \eqref{B(R)}.
\end{lem}
\begin{proof}
	Take any $B \in \mathcal{B}(R)$. By definition, $B$ is positive semidefinite, satisfies conditions \eqref{b_ijij}--\eqref{b_ji}, and fulfills $R_i B_{ij} = R_j$ for all $1 \le i,j \le l$. Because $R'_i$ consists of the first $m'$ rows of $R_i$ and $R'_j$ consists of the first $m'$ rows of $R_j$, taking the first $m'$ rows of the equality $R_i B_{ij} = R_j$ gives $R'_i B_{ij} = R'_j$ for all $i,j$. Hence $B$ satisfies all three conditions required for membership in $\mathcal{B}(R')$, so $B \in \mathcal{B}(R')$. This proves the inclusion $\mathcal{B}(R) \subseteq \mathcal{B}(R')$.
\end{proof}

From \eqref{rank(B - RTR)}, the determination of $\mathcal{R}(G_F)$ reduces to the computation of $\mathrm{rank}(B - R^{T}R)$ for feasible matrices $B$. The following lemma provides a simple relation between this rank and the rank of $B$ itself.

\begin{lem}\label{rank(B)-m}
	For every $B\in\mathcal{B}(R)$ we have
	\[
	\mathrm{rank}(B - R^{T}R)=\mathrm{rank}(B)-\mathrm{rank}(R)=\mathrm{rank}(B)-m.
	\]
\end{lem}
\begin{proof}
	The condition $R_iB_{ij}=R_j$ for all $i,j$ implies the matrix equality $RB = lR$.  Taking transposes yields $BR^{T}=lR^{T}$.  Hence every column of $R^{T}$ is an eigenvector of $B$ with eigenvalue $l$. Let $v_1,\dots,v_m$ denote the columns of $R^{T}$; they span a subspace $V\subseteq\mathbb{R}^{l^{2}}$.
	
	By \eqref{R_qR^T_q} we have $R_jR_j^{T}=I_m$ for each $j$; summing over $j=1,\cdots,l$ gives $RR^{T}=lI_m$.  Therefore $v_i^{T}v_j=l\delta_{ij}$, so the $v_i$ are pairwise orthogonal with norm $\sqrt{l}$.
	Set $u_k:=v_k/\sqrt{l}$ for $k=1,\dots,m$.  Then $\{u_1,\cdots,u_m\}$ is an orthonormal basis of $V$ and satisfies $Bu_k=lu_k$.  
	
	Let $r(B):=\mathrm{rank}(B)$.  Because $B\succeq0$, it admits a spectral decomposition with an orthonormal set of eigenvectors corresponding to its positive eigenvalues.  Explicitly, we may extend $\{u_1,\dots,u_m\}$ to an orthonormal set $\{u_k\}_{k=1}^{r(B)}$ of eigenvectors of $B$ with eigenvalues $\lambda_k>0$ such that
	\[
	B = \sum_{k=1}^{r(B)} \lambda_k u_k u_k^{T}.
	\]
	From the construction above we have $\lambda_1=\cdots=\lambda_m=l$.  Define
	\[
	B_V:=\sum_{k=1}^{m} l\,u_ku_k^{T}, \qquad 
	B_0:=\sum_{k=m+1}^{r(B)} \lambda_k u_k u_k^{T},
	\]
	so that $B=B_V+B_0$ and $B_VB_0=B_0B_V=0$ (since $u_k^{T}u_j=0$ for $k\le m<j$).
	
	Now observe that $R^{T}$ can be expressed as $R^{T}=\sqrt{l}\,(u_1,\cdots,u_m)$.  Hence
	\[
	R^{T}R = \bigl(\sqrt{l}\,(u_1,\cdots,u_m)\bigr)\bigl(\sqrt{l}\,(u_1,\cdots,u_m)\bigr)^{T}
	= l\sum_{k=1}^{m} u_ku_k^{T}=B_V.
	\]
	Therefore
	\[
	B-R^{T}R = (B_V+B_0)-B_V = B_0.
	\]
	Since the supports of $B_V$ and $B_0$ are orthogonal,
	\begin{align*}
		\mathrm{rank}(B-R^{T}R)&=\mathrm{rank}(B_0)\\
		&=\mathrm{rank}(B)-\mathrm{rank}(B_V)=\mathrm{rank}(B)-\mathrm{rank}(R)=\mathrm{rank}(B)-m,
	\end{align*}
	which completes the proof.
\end{proof}

As a consequence of Lemmas~\ref{B(R)inB(R')} and~\ref{rank(B)-m}, we obtain the following corollary.

\begin{cor}\label{cor:r+m-m'}
	Let $\{P_0,\dots,P_m\}$ be a Clifford system on $\mathbb{R}^{2l}$, and let
	$m'<m$. Define
	\[
	G_F(x):=|x|^4-\sum_{\alpha=0}^{m}\langle P_\alpha x,x\rangle^2,
	\qquad
	G'_F(x):=|x|^4-\sum_{\alpha=0}^{m'}\langle P_\alpha x,x\rangle^2 .
	\]
	Then, for any $r\in\mathcal{R}(G_F)$, one has
	$
	r+m-m'\in\mathcal{R}(G'_F).
	$
\end{cor}

\begin{proof}
	Let $\{E_1,\dots,E_{m-1}\}$ be the associated real matrix representation of the
	Clifford algebra induced by the Clifford system $\{P_0,\dots,P_m\}$, and set $E_0=I_l$.
	Let $R$ and $R'$ be the matrices constructed from
	$\{E_1,\dots,E_{m-1}\}$ and $\{E_1,\dots,E_{m'-1}\}$ via
	\eqref{Define R_q} and \eqref{Define R}, respectively.
	Then $R'$ is obtained from $R$ by taking its first $m'$ rows.
	By Lemma~\ref{B(R)inB(R')}, we have
	$
	\mathcal{B}(R)\subseteq\mathcal{B}(R').
	$
	
	If $r\in\mathcal{R}(G_F)$, then there exists $B\in\mathcal{B}(R)$ such that
	$
	\mathrm{rank}(B-R^{T}R)=r
	$
	by \eqref{rank(B - RTR)}.
	By Lemma~\ref{rank(B)-m}, this implies
	$
	\mathrm{rank}(B)=r+m.
	$
	Since $B\in\mathcal{B}(R)\subseteq\mathcal{B}(R')$, applying
	\eqref{rank(B - RTR)} again yields
	\[
	\mathrm{rank}(B-{R'}^{T}R')=\mathrm{rank}(B)-\mathrm{rank}(R')
	=r+m-m'.
	\]
	Hence $r+m-m'\in\mathcal{R}(G'_F)$, completing the proof.
\end{proof}

For a matrix $B = (B_{ij})_{i,j=1}^l \in \mathcal{B}(R)$, Lemma~\ref{properties of B} implies that each diagonal block $B_{ii}$ is the identity matrix $I_l$ and each off‑diagonal block $B_{ik}$ ($i\neq k$) is skew‑symmetric. Because $B$ possesses an $l\times l$ principal submatrix equal to $I_l$, its rank satisfies $\mathrm{rank}(B) \ge l$. The following lemma provides a necessary condition when $\mathrm{rank}(B)=l$.

\begin{lem}\label{rank(B)=l gives Cl algebra}
	Let $B = (B_{ij})_{i,j=1}^l \in \mathcal{B}(R)$. If $\mathrm{rank}(B)=l$, then the blocks satisfy the Clifford relations
	\[
	B_{1i}B_{1j} + B_{1j}B_{1i} = -2\delta_{ij} I_l \qquad (2\le i,j \le l),
	\]
	which implies that $\{B_{12},\dots,B_{1l}\}$ define a representation of the Clifford algebra $C_{l-1}$ on $\mathbb{R}^l$.
	
\end{lem}

\begin{proof}
	Since $B\succeq0$ and $\mathrm{rank}(B)=l$, there exists a matrix $U\in\mathbb{R}^{l\times l^{2}}$ with $\mathrm{rank}(U)=l$ such that $B=U^{T}U$.  Write $U$ in block form as $U=(U_{1}, \cdots, U_{l})$ where each $U_{i}\in\mathbb{R}^{l\times l}$.  Then $B_{ij}=U_{i}^{T}U_{j}$ for all $1\le i,j\le l$.  From $B_{11}=I_{l}$ we obtain $U_{1}^{T}U_{1}=I_{l}$, i.e. $U_{1}$ is orthogonal.
	
	Define $V:=(B_{11}, B_{12}, \cdots, B_{1l})$; this is the matrix formed by the first $l$ rows of $B$.  Because $B_{1j}=U_{1}^{T}U_{j}$, we have
	\[
	V=(U_{1}^{T}U_{1}, U_{1}^{T}U_{2}, \cdots, U_{1}^{T}U_{l})=U_{1}^{T}U .
	\]
	Since $U_{1}$ is orthogonal, $U=U_{1}V$ and consequently
	\begin{equation}\label{eq:B=VTV}
		B=U^{T}U=V^{T}U_{1}^{T}U_{1}V=V^{T}V .
	\end{equation}
	Note that $B_{11}=I_{l}$ and, for $j\ge2$, Lemma~\ref{properties of B} gives $B_{1j}^{T}=-B_{1j}$.  Moreover $B_{jj}=I_{l}$ implies $B_{1j}^{T}B_{1j}=I_{l}$; together with skew‑symmetry this yields $(-B_{1j})B_{1j}=I_{l}$, hence $B_{1j}^{2}=-I_{l}$.  Therefore each $B_{1j}\;(j\ge2)$ is an orthogonal skew‑symmetric matrix.
	
	Now consider $B_{1i}B_{1j}+B_{1j}B_{1i}$ for $2\le i,j\le l$.  Using $B_{1i}^{T}=-B_{1i}$ and the relation $B_{ij}=B_{1i}^{T}B_{1j}$ from \eqref{eq:B=VTV}, we obtain
	\[
	B_{1i}B_{1j}= -B_{1i}^{T}B_{1j}= -B_{ij},\qquad   
	B_{1j}B_{1i}= -B_{1j}^{T}B_{1i}= -B_{ji}.
	\]
	Hence
	\[
	B_{1i}B_{1j}+B_{1j}B_{1i}= -(B_{ij}+B_{ji}).
	\]
	If $i\neq j$, Lemma~\ref{properties of B} tells us $B_{ij}+B_{ji}=0$.  If $i=j$, we already have $B_{1i}^{2}=-I_{l}$, whence $B_{1i}B_{1i}+B_{1i}B_{1i}=2B_{1i}^{2}=-2I_{l}$.  Thus in all cases
	\[
	B_{1i}B_{1j}+B_{1j}B_{1i}= -2\delta_{ij}I_{l}\qquad (2\le i,j\le l),
	\]
	which are precisely the defining relations of the Clifford algebra $C_{l-1}$ on $\mathbb{R}^{l}$.  Therefore $\{B_{12},\dots,B_{1l}\}$ generates a Clifford algebra $C_{l-1}$.
\end{proof}

Equipped with the SDP characterization developed above (especially the description of $\mathcal{R}(G_F)$ via the feasible solutions set $\mathcal{B}(R)$) and the structural lemmas on the matrix $B$, we now turn to a case‑by‑case determination of $\mathcal{R}(G_F)$ for
\[
(m_+,m_-)=(1,k),\;(2,2k-1),\;(3,4),\;(4,3)^I,\;(5,2),\;(6,1),\qquad k\in\mathbb{N}^+ .
\]  
For each of these admissible pairs we shall examine the possible ranks of \emph{sos} representations.  Because $\mathcal{R}(G_F)$ is invariant under geometric equivalence of Clifford systems (Lemma~\ref{R(G_F) invariant}), it suffices to analyse one representative from each geometric equivalence class.  In the following subsections we treat the two infinite families $(1,k)$ and $(2,2k-1)$ and the four remaining cases
$(3,4)$, $(4,3)^I$, $(5,2)$, and $(6,1)$ separately, using the concrete form of the matrices $R$ and the constraints on $B$ to obtain a complete description of $\mathcal{R}(G_F)$.

\subsection{Possible Ranks for $(m_+,m_-)=(3,4),(4,3)^I,(5,2),(6,1)$}\label{m=3,4,5,6}

For the four cases $(m_+,m_-)=(3,4),(4,3)^I,(5,2),(6,1)$, the corresponding values of $m$ are $3, 4, 5, 6$ and  $l$ is always $8$, because $(m_+,m_-)=(m, l-m-1)$.  
By Lemma~\ref{R(G_F) invariant}, which states that $\mathcal{R}(G_F)$ is invariant under geometric equivalence of Clifford systems, it suffices to examine a single Clifford system representation for each case.

In this subsection, we adopt the same Clifford algebra $E_1,\cdots,E_5$ on $\mathbb{R}^8$ and the same Clifford system $P_0,\cdots,P_6$ on $\mathbb{R}^{16}$ as in Subsection~\ref{m=6}.  For each $ m \in \{3,4,5,6\} $, define
\[
G_F^{(m)}(x) := |x|^4 - \sum_{\alpha=0}^{m} \langle P_\alpha x, x \rangle^2,
\]
which is precisely the polynomial $ G_F $ corresponding to the pair $(m,l)=(m,8)$. 

Let $E_0=I_8$. For $m=3,4,5,6$ and $1\leq q\leq 8$, let $R^{(m)}_q$ be the matrix obtained from $\{E_0,\cdots,E_{m-1}\}$ via Definition \eqref{Define R_q}; and let $R^{(m)}:=(R^{(m)}_1,\cdots,R^{(m)}_8)$ (note that $R^{(6)}_q$ and $R^{(6)}$ are the same as defined in \eqref{Define R_q^6}). By definition, $R^{(m)}$ is the submatrix of $R^{(6)}$ consisting of its first $m$ rows. Consequently, Lemma \ref{B(R)inB(R')} yields the chain of inclusions
\begin{equation}\label{chain of inclusions}
	\mathcal{B}(R^{(6)}) \subseteq \mathcal{B}(R^{(5)}) \subseteq \mathcal{B}(R^{(4)}) \subseteq \mathcal{B}(R^{(3)}).
\end{equation}
In Subsection~\ref{m=6} we have shown that $\mathcal{B}(R^{(6)})=\{B^{(6)}\}$, where $B^{(6)}$ is defined in \eqref{Define B^6}. Next we show that $\mathcal{B}(R^{(3)})$ likewise consists of a single element; that is, the following SDP for the matrix $B = (B_{ij})_{i,j=1}^{8}$ admits a unique solution:
\begin{equation}\label{SDP for m=3}
	\begin{cases}
		B \succeq 0, \\[2pt]
		R^{(3)}_{i} B_{ij} = R^{(3)}_{j}, \quad 1 \le i, j \le 8, \\[2pt]
		\text{conditions } \eqref{b_ijij}\text{--}\eqref{b_ji} \text{ hold}.
	\end{cases}
\end{equation}

As in Subsections \ref{m=3 l=4r} and \ref{m=6}, the solution of $\mathcal{B}(R^{(3)})$ is obtained analogously; we outline it briefly. Recall that $\{v_q\}_{q=1}^l \subset \mathbb{R}^l$ and $\{w_{\alpha}\}_{\alpha=1}^m \subset \mathbb{R}^m$ are the standard basis row vectors. Computing the second and third rows of $R^{(3)}_{1} B_{1j} = R^{(3)}_{j}$, the third row of $R^{(3)}_{2} B_{2j} = R^{(3)}_{j}$, the second and third rows of $R^{(3)}_{4} B_{4j} = R^{(3)}_{j}$, and the second row of $R^{(3)}_{5} B_{5j} = R^{(3)}_{j}$ yields  
\begin{align*}
	v_2 B_{1j} &= - w_{2} R^{(3)}_{j}, &  v_7 B_{1j} &= w_{3} R^{(3)}_{j}, &  v_8 B_{2j} &= w_{3} R^{(3)}_{j},\\
	v_3 B_{4j} &= w_{2} R^{(3)}_{j}, &  v_6 B_{4j} &= w_{3} R^{(3)}_{j}, &  v_6 B_{5j} &= w_{2} R^{(3)}_{j}.
\end{align*}
By Lemma \ref{Btauthm}, we obtain
\begin{align}
	B_{12} &= \tau_{1}(E_{1}), &
	B_{17} &= - \tau_{1}(E_{2}), &
	B_{28} &= - \tau_{2}(E_{2}), \label{B12,17,28}\\
	B_{43} &= - \tau_{4}(E_{1}), &
	B_{46} &= - \tau_{4}(E_{2}), &
	B_{56} &= - \tau_{5}(E_{1}),\label{B43,46,56}
\end{align}
all of which are orthogonal matrices.

Lemma \ref{properties of B} gives $B_{ii}=I_l$, and $B_{ik}$ is skew-symmetric with $B_{ki}=-B_{ik}$ for $i \neq k$. Since 
\[
R^{(3)}_{1} B_{15} = R^{(3)}_{5}\quad \text{and}\quad R^{(3)}_{5} B_{15} = -R^{(3)}_{5} B_{51}=-R^{(3)}_{1},
\] 
the 1st, 2nd, 3rd, 5th, 6th, and 7th rows of $B_{15}$ are completely determined. Moreover, by the skew-symmetry of $B_{15}$, only the $(4,8)$ and $(8,4)$ entries of $B_{15}$ remain undetermined. Denote the $(4,8)$ entry by $d$; then the $(8,4)$ entry is $-d$. 

On the other hand, the relation 
$$R^{(3)}_{6} B_{16} = -R^{(3)}_{6} B_{61}=-R^{(3)}_{1}$$ 
yields $v_4 B_{16} =  w_{3} R^{(3)}_{1}$, that is, the fourth row of $B_{16}$ equals $ w_{3} R^{(3)}_{1}$. Since $B_{56}$ is an orthogonal matrix, by Lemma \ref{BBthm} we have
\[
B_{15}=-B_{51}=-B_{56}B_{61}=B_{56}B_{16}.
\]
Thus,
\[
-d=(v_{8}B_{56}) (-v_{4}B_{16})^T=(v_{8}\tau_{5}(E_{1})) (w_{3} R^{(3)}_{1})^{T}=(-v_7) (v_7)^{T}=-1.
\]
Then $d=1$, and we have now completely determined the matrix $B_{15}$:
\[
B_{15} = \begin{pmatrix}
	0 & 0 & 0 & 0 & 1 & 0 & 0 & 0 \\
	0 & 0 & 0 & 0 & 0 & -1 & 0 & 0 \\
	0 & 0 & 0 & 0 & 0 & 0 & 1 & 0 \\
	0 & 0 & 0 & 0 & 0 & 0 & 0 & 1 \\
	-1 & 0 & 0 & 0 & 0 & 0 & 0 & 0 \\
	0 & 1 & 0 & 0 & 0 & 0 & 0 & 0 \\
	0 & 0 & -1 & 0 & 0 & 0 & 0 & 0 \\
	0 & 0 & 0 & -1 & 0 & 0 & 0 & 0
\end{pmatrix}=-\tau_{1}(E_{4}).
\]
$B_{15}$ is orthogonal, and the six matrices in \eqref{B12,17,28} and \eqref{B43,46,56} are also orthogonal. Hence, by Lemma \ref{BBthm}, we obtain
\[
B_{16}=B_{15}B_{56},\quad B_{14}=-B_{16}B_{46},\quad B_{13}=B_{14}B_{43},\quad B_{18}=B_{12}B_{28}.
\]
This implies that all $B_{1i}$ ($1\leq i\leq 8$) are orthogonal matrices. Consequently, for $1\leq i,j\leq 8$,  
\[
B_{ij}=B_{i1}B_{1j}=B_{1i}^TB_{1j},\quad 1\leq i,j\leq 8.
\]

Thus, the SDP \eqref{SDP for m=3} has been shown to have a unique solution; i.e., the set $\mathcal{B}(R^{(3)})$ consists of a single element. Applying the inclusion relations in \eqref{chain of inclusions} yields
\[
\mathcal{B}(R^{(3)}) = \mathcal{B}(R^{(4)}) = \mathcal{B}(R^{(5)}) = \mathcal{B}(R^{(6)}) = \{B^{(6)}\}.
\]
From \eqref{rank(B - RTR)} and Lemma \ref{rank(B)-m}, it follows that
\[
\mathcal{R}(G_F^{(m)})=\Big\{\mathrm{rank}\(B^{(6)} - (R^{(m)})^{T}(R^{(m)})\)\Big\}=\{8-m\}.
\]
In summary, let $r$ denote the rank of any \emph{sos} representation of $G_F$. Then:
\begin{enumerate}
	\item For $(m_+, m_-) = (3,4)$, $r = 8 - 3 = 5$.
	\item For $(m_+, m_-) = (4,3)^I$, $r = 8 - 4 = 4$.
	\item For $(m_+, m_-) = (5,2)$, $r = 8 - 5 = 3$.
	\item For $(m_+, m_-) = (6,1)$, $r = 8 - 6 = 2$.
\end{enumerate}

\subsection{Possible Ranks for $(m_+,m_-)=(1,k)$}\label{sos-rank for m=1}

As discussed in Subsection~\ref{sos (1,k+2)}, we recall the case $(m,l) = (1,k+2)$. 
Here the matrices $P_0$ and $P_1$ are fixed constant matrices (see \eqref{Cliffordsys-alg}). 
We consider
\[
G_F(x) = |x|^4 - \sum_{\alpha=0}^{1} \langle P_\alpha x, x \rangle^2 .
\]
In this section, we write $R_q$ and $R$ for $R_q(1,l)$ and $R(1,l)$, respectively.
Subsection~\ref{sos (1,k+2)} shows that $\mathcal{B}(R)$ is nonempty and constructs an element $B(1,l) \in \mathcal{B}(R)$, which in turn implies that $G_F$ is \emph{sos}.
 
Let $B \in \mathcal{B}(R)$ satisfy conditions \eqref{b_ijij}--\eqref{b_ji} and write
\[
B=\bigl(b_{ij,kh}\bigr)_{l^2\times l^2}=\bigl(B_{ik}\bigr)_{i,k=1}^{l}.
\]
We claim that
\[
l \le \mathrm{rank}(B) \le \frac{l(l-1)}{2}+1.
\]

	By Lemma~\ref{properties of B}, one has $B_{ii}=I_l$ for every $i$. Hence, for each $i$, the diagonal block $B_{ii}$ is an $l\times l$ principal submatrix of $B$ and is nonsingular. Therefore,
	\[
	\mathrm{rank}(B)\ge \mathrm{rank}(B_{ii})=l.
	\]
	For the upper bound, index the rows of $B$ by ordered pairs $(i,j)$ with $1\le i,j\le l$, and denote by $\rho_{ij}\in\mathbb{R}^{1\times l^2}$ the $(i,j)$-row of $B$ in this indexing (equivalently, the row corresponding to the $j$-th row of the $i$-th block row).
	Condition \eqref{b_iikk} implies that all diagonal rows coincide, namely
	\[
	\rho_{11}=\rho_{22}=\cdots=\rho_{ll}.
	\]
	Moreover, by \eqref{b_ji}, for $i\neq j$ the off-diagonal rows satisfy
	\[
	\rho_{ji}=-\rho_{ij}.
	\]
	Consequently, the row space of $B$ is spanned by the single row $\rho_{11}$ together with the rows $\rho_{ij}$ for $1\le i<j\le l$. 
	Since $\mathrm{rank}(B)=\dim(\mathrm{Row}(B))$, we obtain
	\[
	\mathrm{rank}(B)\le \frac{l(l-1)}{2}+1,
	\]
	as desired.
	
By Proposition~\ref{prop:B(1,l)}, the upper bound of $\mathrm{rank}(B)$ is attained, for instance, when $B=B(1,l)$.
We now turn to the characterization of the equality case
$\mathrm{rank}(B)=l$ for the lower bound.

Assume that $\mathrm{rank}(B)=l$.
By Lemma~\ref{rank(B)=l gives Cl algebra},
the matrices $\{B_{12},\dots,B_{1l}\}$ define a representation of the Clifford
algebra $C_{l-1}$ on $\mathbb{R}^l$.
In particular, $C_{l-1}$ admits a real representation on $\mathbb{R}^l$.
On the other hand, the minimal dimension of an irreducible real representation
of $C_{l-1}$ is given by $\delta(l)$
(see Table~\ref{representation of Cl}).
Since $l\geq 3$, the condition $\mathrm{rank}(B)=l$ can only occur when
$l=4$ or $l=8$.
We now examine these two cases separately.

\noindent
\emph{Case $l=4$.}
By Lemma~\ref{B(R)inB(R')},
\[
B(2,4)\in \mathcal{B}(R(2,4)) \subseteq \mathcal{B}(R(1,4)).
\]
Moreover, Proposition~\ref{prop:B(2,l)} shows that
\[
\mathrm{rank}\bigl(B(2,4)\bigr)=4,
\]
which attains the lower bound.

\noindent
\emph{Case $l=8$.}
By Lemma~\ref{B(R)inB(R')},
\[
B^{(6)}\in \mathcal{B}(R(6,8)) \subseteq \mathcal{B}(R(1,8)).
\]
From the definition \eqref{Define B^6}, it is immediate that
\[
\mathrm{rank}\bigl(B^{(6)}\bigr)=8,
\]
which again attains the lower bound.

\noindent
Consequently, the equality $\mathrm{rank}(B)=l$ can occur if and only if
$l=4$ or $l=8$.

From \eqref{rank(B - RTR)} and Lemma \ref{rank(B)-m}, we have
\[
\mathcal{R}(G_F)
=\Big\{\mathrm{rank}\bigl(B-R^{T}R\bigr):B\in\mathcal{B}(R)\Big\}
=\{\mathrm{rank}(B)-1:B\in\mathcal{B}(R)\}.
\]
Let $r$ denote the rank of an arbitrary \emph{sos} representation of $G_F$. Then $r\in\mathcal{R}(G_F)$ and hence
\[
l-1\le r\le \frac{l(l-1)}{2}.
\]
Moreover, the upper bound is attainable, for instance by taking $B=B(1,l)$.
Finally, the lower bound $r=l-1$ is attainable if and only if $l=4$ or $l=8$.

\subsection{Possible Ranks for $(m_+,m_-)=(2,2k-1)$}\label{sos-rank for m=2}

By Lemma~\ref{R(G_F) invariant}, which states that $\mathcal{R}(G_F)$ is invariant
under geometric equivalence of Clifford systems, it suffices to examine
a single Clifford system representation in the case $(m_+,m_-)=(2,2k-1)$.

As discussed in Subsection~\ref{sos (2,2k+2)}, we recall the case $(m,l) = (2,2k+2)$.
Here the matrices $P_0$, $P_1$, and $P_2$ are fixed constant matrices chosen
as in Subsection~\ref{sos (2,2k+2)}.
We consider
\[
G_F(x) = |x|^4 - \sum_{\alpha=0}^{2} \langle P_\alpha x, x \rangle^2 .
\]
In this section, we write $R_q$ and $R$ for $R_q(2,l)$ and $R(2,l)$, respectively.
Subsection~\ref{sos (2,2k+2)} shows that $\mathcal{B}(R)$ is nonempty and constructs
an element $B(2,l) \in \mathcal{B}(R)$, which in turn implies that $G_F$ is
\emph{sos}.

Let $B \in \mathcal{B}(R)$ satisfy conditions \eqref{b_ijij}--\eqref{b_ji} and write
\[
B=\bigl(b_{ij,kh}\bigr)_{l^2\times l^2}
=\bigl(B_{ik}\bigr)_{i,k=1}^{l}.
\]
Then
\begin{equation}\label{range of rank(B) of m=2}
	l \le \mathrm{rank}(B) \le \frac{l(l-2)}{4}+2.
\end{equation}

By Lemma~\ref{properties of B}, one has $B_{ii}=I_l$ for every $i$.
Hence, for each $i$, the diagonal block $B_{ii}$ is an $l\times l$
principal submatrix of $B$ and is nonsingular. Therefore,
\[
\mathrm{rank}(B)\ge \mathrm{rank}(B_{ii})=l.
\]

We now prove the upper bound.
Index the rows of $B$ by ordered pairs $(i,j)$ with $1\le i,j\le l$,
and denote by $\rho_{ij}\in\mathbb{R}^{1\times l^2}$ the $(i,j)$-row of $B$
in this indexing.

As shown in Subsection~\ref{sos (2,2k+2)}, for each $1\le s\le k+1$
the block $B_{2s-1,2s}=-\tau_{2s}(E_1)$ is an orthogonal matrix. Moreover, by Lemma~\ref{BBthm} and the skew-symmetry relation
$B_{2s,2s-1}=-B_{2s-1,2s}$, for every $1\le j\le l$ one has
\[
B_{2s,j}
= B_{2s,2s-1} B_{2s-1,j}
= -\, B_{2s-1,2s} B_{2s-1,j}.
\]
For each such $s$, we left-multiply the $(2s-1)$-st block row of $B$
by $B_{2s-1,2s}$ and add it to the $2s$-th block row.
These elementary row operations eliminate all even block rows of $B$.
Consequently, the row space of $B$ is spanned by at most $l^2/2$ rows.

On the other hand, by \eqref{R2s-1} we have
\[
R_{2s-1}\bigl(B_{2s-1,1},\cdots,B_{2s-1,l}\bigr)
=(R_1,\cdots,R_l)=R
\]
for every $s$.
By the explicit construction in Section~\ref{sos (2,2k+2)},
one has
$
R_{2s-1}=R_1 L_s,
$
where $R_1$ and $L_s$ are given there.
Thus, for each $s$ the rows $\rho_{2s-1,2s-1}$ and
$\rho_{2s-1,2s}$ coincide with the first and second rows of $R$,
respectively.
Equivalently, one has
\[
\rho_{11}=\rho_{33}=\cdots=\rho_{2k+1,2k+1}=w_1R,
\qquad
\rho_{12}=\rho_{34}=\cdots=\rho_{2k+1,2k+2}=w_2R,
\]
where $w_1=(1,0), w_2=(0,1)\in\mathbb{R}^2$. Since $R$ has full row rank, the above relations impose independent
affine constraints on the row space of $B$.

Moreover, by \eqref{b_ji}, for $i\neq j$ the off-diagonal rows satisfy
\[
\rho_{ji}=-\rho_{ij}.
\]
Therefore, after removing the $l/2$ identical rows
$\rho_{2s-1,2s-1}$ and the $l/2$ identical rows $\rho_{2s-1,2s}$,
and taking into account the skew-symmetry $\rho_{ji}=-\rho_{ij}$,
the dimension of the row space is bounded by
\[
\frac{1}{2}\Bigl(\frac{l^2}{2}-l\Bigr)+2=\frac{l(l-2)}{4}+2.
\]
Hence,
\[
\mathrm{rank}(B)\le \frac{l(l-2)}{4}+2,
\]
as claimed.

By Proposition~\ref{prop:B(2,l)}, the upper bound of $\mathrm{rank}(B)$
is attained when $B=B(2,l)$.
In particular, when $l=4$, the upper and lower bounds coincide, and hence
$
\mathrm{rank}(B)=4.
$
In this case, the matrix $B(2,4)$ realizes this value.

We now turn to the characterization of the equality case
$\mathrm{rank}(B)=l$ for the lower bound when $l>4$.
Assume that $\mathrm{rank}(B)=l$ with $l\ge 5$.
By Lemma~\ref{rank(B)=l gives Cl algebra},
the matrices $\{B_{12},\dots,B_{1l}\}$ define a representation of the Clifford
algebra $C_{l-1}$ on $\mathbb{R}^l$.
In particular, $C_{l-1}$ admits a real representation on $\mathbb{R}^l$.
On the other hand, the minimal dimension of an irreducible real representation
of $C_{l-1}$ is given by $\delta(l)$
(see Table~\ref{representation of Cl}).
It follows that the condition $\mathrm{rank}(B)=l$ with $l>4$
can occur only when $l=8$.

For $l=8$, we emphasize that the present situation is different from the case
$(m,l)=(1,8)$.
In particular, Lemma~\ref{B(R)inB(R')} cannot be applied directly to relate
$\mathcal{B}(R(6,8))$ and $\mathcal{B}(R(2,8))$, since the second row of $R(2,8)$
does not coincide with that of $R(6,8)$, and hence the assumptions of
Lemma~\ref{B(R)inB(R')} are not satisfied. Let $\{P_0',\dots,P_6'\}$ be a Clifford system on $\mathbb{R}^{16}$, and define
\[
G_F'(x):=|x|^4-\sum_{\alpha=0}^{6}\langle P_\alpha' x,x\rangle^2,
\qquad
G_F''(x):=|x|^4-\sum_{\alpha=0}^{2}\langle P_\alpha' x,x\rangle^2 .
\]
As shown in Subsection~\ref{m=3,4,5,6}, one has $\mathcal{R}(G_F')=\{2\}$.
It then follows from Corollary~\ref{cor:r+m-m'} that
$
6\in\mathcal{R}(G_F'').
$
Since the Clifford systems $\{P_0,P_1,P_2\}$ and $\{P_0',P_1',P_2'\}$ are
geometrically equivalent, Lemma~\ref{R(G_F) invariant} implies that
$
\mathcal{R}(G_F)=\mathcal{R}(G_F''),
$
and hence $6\in\mathcal{R}(G_F)$.
Therefore, there exists $B\in\mathcal{B}(R)$ such that
$
\mathrm{rank}(B-R^{T}R)=6
$
by \eqref{rank(B - RTR)}.
Applying Lemma~\ref{rank(B)-m}, we obtain
\[
\mathrm{rank}(B)=\mathrm{rank}(B-R^{T}R)+\mathrm{rank}(R)=8=l.
\]
Consequently, the lower bound in \eqref{range of rank(B) of m=2} is attainable
when $l=8$.

From \eqref{rank(B - RTR)} and Lemma~\ref{rank(B)-m}, we have
\[
\mathcal{R}(G_F)
=\Big\{\mathrm{rank}\bigl(B-R^{T}R\bigr):B\in\mathcal{B}(R)\Big\}
=\{\mathrm{rank}(B)-2:B\in\mathcal{B}(R)\}.
\]
Let $r$ denote the rank of an arbitrary \emph{sos} representation of $G_F$.
Then $r\in\mathcal{R}(G_F)$ and hence
\[
l-2\le r\le \frac{l(l-2)}{4}.
\]
Moreover, the upper bound is attainable, for instance by taking $B=B(2,l)$.
Finally, the lower bound $r=l-2$ is attainable if and only if $l=4$ or $l=8$.

Combining the case-by-case analysis in Subsections~\ref{m=3,4,5,6}, \ref{sos-rank for m=1}, and~\ref{sos-rank for m=2},
the proof of Theorem~\ref{rank thm} is now complete.



\begin{thebibliography}{123}
	
\bibitem{Cecil07}
Thomas E. Cecil, Quo-Shin Chi, Gary R. Jensen.
Isoparametric hypersurfaces with four principal curvatures.
\emph{Ann. of Math.}, 166: 1--76, 2007.

\bibitem{Cecil15}
Thomas E. Cecil, Patrick J. Ryan.
\emph{Geometry of Hypersurfaces}.
Springer Monographs in Mathematics. New York: Springer, 2015.

\bibitem{Chi11}
Quo-Shin Chi.
Isoparametric hypersurfaces with four principal curvatures, II.
\emph{Nagoya Math. J.}, 204: 1--18, 2011.

\bibitem{Chi13}
Quo-Shin Chi.
Isoparametric hypersurfaces with four principal curvatures, III.
\emph{J. Differential Geom.}, 94(3): 469--504, 2013.

\bibitem{Chi20}
Quo-Shin Chi.
Isoparametric hypersurfaces with four principal curvatures, IV.
\emph{J. Differential Geom.}, 115: 225--301, 2020.

\bibitem{Fa17}
Fuquan Fang.
Dual submanifolds in rational homology spheres.
\emph{Sci. China Math.}, 60(9): 1549--1560, 2017.

\bibitem{FKM81}
Dirk Ferus, Hermann Karcher, Hans-Friedrich M\"unzner.
Cliffordalgebren und neue isoparametrische Hyperfl\"achen.
\emph{Math. Z.}, 177: 479--502, 1981.

\bibitem{GeQianTangYan25}
Jianquan Ge, Chao Qian, Zizhou Tang, Wenjiao Yan.
An overview of the development of isoparametric theory (in Chinese).
\emph{Sci. Sin. Math.}, 55: 145--168, 2025.

\bibitem{GT13}
Jianquan Ge, Zizhou Tang.
Isoparametric functions and exotic spheres.
\emph{J. Reine Angew. Math.}, 683: 161--180, 2013.

\bibitem{GT23}
Jianquan Ge, Zizhou Tang.
Isoparametric polynomials and sums of squares.
\emph{Int. Math. Res. Not. IMRN}, 24: 21226--21271, 2023.

\bibitem{Miyaoka13}
Reiko Miyaoka.
Isoparametric hypersurfaces with $(g,m)=(6,2)$.
\emph{Ann. of Math. (2)}, 177: 53--110, 2013.

\bibitem{Miyaoka16}
Reiko Miyaoka.
Errata of ``Isoparametric hypersurfaces with $(g,m)=(6,2)$''.
\emph{Ann. of Math. (2)}, 183: 1057--1071, 2016.

\bibitem{Mun}
Hans-Friedrich M\"unzner.
Isoparametrische Hyperfl\"achen in Sph\"aren, I and II.
\emph{Math. Ann.}, 251: 57--71, 1980 and 256: 215--232, 1981.

\bibitem{OT75}
Hideki Ozeki, Masaru Takeuchi.
On some types of isoparametric hypersurfaces in spheres, I and II.
\emph{Tohoku Math. J.}, 27: 515--559, 1975 and 28: 7--55, 1976.

\bibitem{Lall03}
Pablo A. Parrilo, Sanjay Lall.
Semidefinite programming relaxations and algebraic optimization in control.
\emph{European J. Control}, 9(2--3): 307--321, 2003.

\bibitem{Lall09}
Antonis Papachristodoulou, Matthew M. Peet, Sanjay Lall.
Analysis of polynomial systems with time delays via the sum of squares decomposition.
\emph{IEEE Trans. Automat. Control}, 54(5): 1058--1064, 2009.

\bibitem{QT15}
Chao Qian, Zizhou Tang.
Isoparametric functions on exotic spheres.
\emph{Adv. Math.}, 272: 611--629, 2015.

\bibitem{Solomon92}
Bruce Solomon.
Quartic isoparametric hypersurfaces and quadratic forms.
\emph{Math. Ann.}, 293(3): 387--398, 1992.

\bibitem{TangYan13}
Zizhou Tang, Wenjiao Yan.
Isoparametric foliation and Yau conjecture on the first eigenvalue.
\emph{J. Differential Geom.}, 94(3): 521--540, 2013.

\bibitem{TangXieYan14}
Zizhou Tang, Yuquan Xie, Wenjiao Yan.
Isoparametric foliation and Yau conjecture on the first eigenvalue, II.
\emph{J. Funct. Anal.}, 266: 6174--6199, 2014.

\bibitem{Wilson2021}
Robert Arnott Wilson.
On the problem of choosing subgroups of Clifford algebras for applications in fundamental physics.
\emph{Adv. Appl. Clifford Algebras}, 31: 59, 2021.


	
\end{thebibliography}
\end{document}